\documentclass[12pt,reqno]{amsart}

\usepackage{amsthm, amsmath, amscd, amsfonts}
\usepackage{cases}
\usepackage{nameref, hyperref, cleveref}
\usepackage{fancyhdr}
\usepackage{array}
\usepackage[normalem]{ulem}
\usepackage{caption}
\usepackage{subcaption}
\usepackage{graphicx}
\usepackage[usenames,dvipsnames]{xcolor}
\usepackage{tikz}
\usepackage[color=WildStrawberry]{todonotes}
\usepackage{geometry}
\usepackage{multicol}
\usepackage{mathabx}
\usepackage[utf8]{inputenc}

\title{Sections and Chapters}

\usetikzlibrary{shapes}

\newcommand{\Z}{\mathbb{Z}}

\newcommand{\rk}{\textup{rk }}

\newcommand{\spl}[2]{{\bf #1}_{#2}}

\newcommand{\lcm}{\textup{lcm}}

\renewcommand{\a}{\alpha}

\newtheorem{theorem}{Theorem}[section]
\newtheorem{corollary}[theorem]{Corollary}
\newtheorem{lemma}[theorem]{Lemma}
\newtheorem{proposition}[theorem]{Proposition}

\newtheorem{definition}[theorem]{Definition}
\newtheorem{remark}[theorem]{Remark}
\newtheorem{example}[theorem]{Example}
\newtheorem{algorithm}[theorem]{Algorithm}
\newtheorem{question}[theorem]{Question}

\newtheorem*{theorem*}{Theorem}

\definecolor{mygray}{cmyk}{0.1875,0,0.375,0.8745}
\definecolor{myraspberry}{cmyk}{0,1,0.15,0.3}
\definecolor{mysunshine}{cmyk}{0,0.104,0.542,0.0157}
\definecolor{mypurple}{cmyk}{0.324,0.743,0,0.488}

	\tikzstyle{rectanglevertex}=
	[rounded corners=3pt,
	line width=1pt,
	draw=mygray,
	fill=black!5, 
	font=\fontsize{11}{11}]
	\tikzstyle{vertex}=[rounded corners=3,line width=1pt, draw=black!50,fill=white, font=\fontsize{12}{12}\boldmath\rmfamily\bfseries]
	\tikzstyle{edge}=[line width=1pt,]
	\tikzstyle{edgelabel}=[font=\fontsize{12}{12}\boldmath\rmfamily\bfseries,color=myraspberry]
	\tikzstyle{dashedarrow}=[dashed, line width=4pt, ->]

\tikzstyle{dashededge}=[line width=1.5pt,dotted]

\title{Splines mod $m$}
\author[Bowden and Tymoczko]{Nealy Bowden and Julianna Tymoczko}

\address{Smith College, Northampton, MA}
\thanks{The first author was partially supported by NSF grant DMS--1143716.  The second author was partially supported by NSF grant DMS--1248171}
\email{jtymoczko@smith.edu}

\begin{document}

\begin{abstract}
Given a graph whose edges are labeled by ideals in a ring, a generalized spline is a labeling of each vertex by a ring element so that adjacent vertices differ by an element of the ideal associated to the edge.  We study splines over the ring $\Z/m\Z$. Previous work considered splines over domains, in which very different phenomena occur.  For instance when the ring is the integers, the elements of bases for spline modules are indexed by the vertices of the graph.  However we prove that over $\Z/m\Z$ spline modules can essentially have any rank between $1$ and $n$.  Using the classification of finite $\Z$-modules, we begin the work of classifying splines over $\Z/m\Z$ and produce minimum generating sets for splines on cycles over $\Z/m\Z$. We close with many  open questions.

\end{abstract}

\maketitle

\clearpage

\tableofcontents
\listoffigures

\clearpage
\section{Introduction}

This paper considers splines over the integers mod $m$.  For us the $\Z$-module of splines is parametrized by a graph $G$ each of whose edges is labeled with an element of $\Z/m\Z$.  A spline on an edge-labeled graph $G=(V, E, \alpha)$ is an element $\spl{f}{} \in (\Z/m\Z)^{|V|}$ so that for each edge $uv$ in the graph, the difference $\spl{f}{u}-\spl{f}{v}$ is a multiple of the label on the edge $uv$. In other words   $\spl{f}{u}-\spl{f}{v} \in \left< uv \right>$.

The reader familiar with the classical definition of splines may feel bemused.  Classically splines are defined as the module of polynomials over the faces of a given polytope, with the condition  that the polynomials agree up to a specific order at intersections of the faces.  One important example is the collection of piecewise-polynomials over a polytope.  Splines occur throughout mathematics: applied mathematicians use them to approximate complicated functions or isolated data points by relatively manageable functions; analysts classify splines in low dimensions or with other fixed parameters \cite{AlfSch87, AlfSch90, Sch84a, Sch84b, Alf86}; algebraists study algebraic invariants of spline modules  \cite{Bil88, BilRos91, BilRos92, DiP12, DiP14, Haa91, Ros95, Ros04, Yuz92}; and geometers and topologists use splines to describe the equivariant cohomology ring of well-behaved geometric objects \cite{GKM98, Pay06, BFR09, Sch12}.  Recent work generalizes splines to a more abstract algebraic and combinatorial setting, of which the definition in the previous paragraph is still just a special case (see also Section \ref{section: definitions}) \cite{GPT, HMR, BHKR}.

Most existing work considers splines over specific rings: in the traditional analytic and applied cases, real polynomials rings; in the geometric/topological case and algebraic case, often complex polynomial rings; in a few cases, integers.  We consider splines over integers mod $m$.  Our motivation for studying these rings comes from Braden-MacPherson's construction of intersection homology as a module similar to generalized splines \cite{BraMac01}.   Quotient rings and sums of quotient rings appear frequently in Braden-MacPherson's construction.  We view the work in this paper in part as a laboratory for those more complicated settings.

The key difference between previous work and this paper is that $\Z/m\Z$ is not a domain.  Very different phenomena emerge in splines whose base ring is not a domain, and in particular whose base ring is $\Z/m\Z$.  First these spline modules are {\em finite}.  Thus  spline modules over $\Z/m\Z$ must have {\em minimum} generating sets---namely a generating set with smallest possible size (which is different from a {\em minimal} generating set over a ring that is not a domain).  In general spline modules need not be free \cite{Sch12, DiP12}.   The structure theorem for finite abelian groups shows that finite modules are generally not free, but the minimum generating sets function like bases except that each element $\spl{b}{}$ of the minimum generating set has a scalar $c_b$ with $c_b \spl{b}{}=0$.  For this reason these minimum generating sets of finite abelian groups are sometimes called bases in the literature \cite{Dou51a, Dou51b, Dou51c, Dou51d}; see  Proposition \ref{proposition: basis exists} and the surrounding discussion for more.  

Over $\Z/m\Z$ these minimum generating sets can be smaller than expected.  Over a domain we know that the module of splines contains a free submodule of rank at least the number of vertices \cite{GPT}, and over a PID the module of splines is always free with rank the number of vertices \cite{Ros, HagTym}. Theorem \ref{theorem: max rank is n} shows that there are at most $n$ elements in the minimum generating set for splines mod $m$ on a graph with $n$ vertices.  The rank of the $\Z$-module of splines is defined to be the number of elements of a minimum generating set (not to be confused with the free rank of a $\Z$-module \cite[Definition 1, Page 165]{DumFoo03}).  One of our main results proves the module of splines can have essentially any rank over a ring with zero divisors.  More precisely Theorem \ref{edge and set size theorem} says the following.  
\begin{theorem*}
Fix $n$ and $m$ so that $n\geq 4$ and $m$ has at least $2$ prime factors or so that $n=3$ and $m$ has at least $3$ prime factors.  Then there is a graph with $n$ vertices whose splines over $\Z/m\Z$ have rank $k$ for each $k=1,2,\ldots,n$.
\end{theorem*}

Theorem \ref{structure theorem} gives another main result: a structure theorem for splines mod $m$.
\begin{theorem*}
Let $G$ be a graph whose edges are labeled with elements in $\Z/m\Z$.  Let $m',m''$ be relatively prime integers with $m'm''=m$ and let $G'$ be the graph obtained by taking each edge of $G$ mod $m'$ (respectively $G''$, $m''$).  Then the rings of splines $R_G, R_{G'},$ and $R_{G''}$ are related by
\[R_G \cong R_{G'} \oplus R_{G''}\]
\end{theorem*}
In other words we can use the prime factorization of $m$ to reduce computations of splines considerably.  This is similar to the technique of localizing at prime ideals in a polynomial ring used in work on algebraic splines, first by Billera and Rose \cite[Theorem 2.3]{BilRos91} and later by Yuzvinsky \cite[Lemma 2.3]{Yuz92} and DiPasquale \cite[Proposition 3.5 and Corollary 3.6]{DiP14} Section \ref{section: classification} uses Theorem  \ref{structure theorem} and other results in this paper to classify the module of splines completely for $m=p$ and $m=p^2$ and any graph $G$, as well as for arbitrary $m$ and graphs $G=C_n$ that are cycles.   Algorithm \ref{algorithm} and Theorem \ref{theorem: algorithm works} explicitly construct a minimum generating set for splines mod $m$ over cycles, consisting of an analogue of upper-triangular basis elements called {\em flow-up splines}.

Section \ref{questions section} concludes with a number of open questions.

\section{Notation and background}

In this section we give the general definition of splines, as well as the special case of the definition used in this manuscript.  We also describe an analogue of upper-triangular bases for splines, which are called {\em flow-up splines} because of geometric applications in which the elements are defined by certain torus flows \cite{GolTol09, KnuTao03, Tym05, HMR}.  The section ends by using the structure theorem for finite abelian groups to give conditions for when flow-up splines also form a minimum generating set. 

\subsection{Definitions and Notation}\label{section: definitions}

\begin{definition}
Let $G=(V,E)$ be a finite graph.  Let $R$ be a commutative ring with identity.  Let $\alpha: E \rightarrow \{\textup{ideals in }R\}$ be a function that labels the edges of $G$ with ideals in $R$.  The splines on $G$ are elements $\spl{f}{} \in R^{|V|}$ such that for each edge $uv \in E$ we have
\[\spl{f}{u} - \spl{f}{v} \in \left<\alpha({uv})\right>.\]
The collection of splines over the graph $G$ with edge-labeling $\alpha$ is denoted $R_{G,\alpha}$ or just $R_G$ if the edge-labeling is clear.
\end{definition}

The collection of splines $R_{G,\alpha}$ form a ring and an $R$-module with the subring and submodule structure inherited from $R^{|V|}$.  In particular the identity spline $\spl{1}{} \in R_{G,\alpha}$ is defined so that $\spl{1}{v}=1$ for all $v \in V$ and the zero spline $\spl{0}{} \in R_{G,\alpha}$ is defined so that $\spl{0}{v}=0$ for all $v \in V$.  

In this paper the base ring is the quotient ring $R=\Z/m\Z$. Every ideal in $\Z/m\Z$ is principal so we typically describe an edge-label $\langle a \rangle$ by the generator $a \in \Z/m\Z$ as in the introduction. 

{\bf Throughout this manuscript $m$ refers to the modulus and $n$ refers to the number of vertices in $G$.}

The graph we discuss most in this manuscript is the cycle with $n$ vertices, which we label as shown in Figure \ref{ncyclelabels}.
\begin{figure}[h]
\begin{center}
\begin{tikzpicture}
	\pgfmathsetmacro{\r}{2.5}
	\pgfmathsetmacro{\ro}{2.75}
	\pgfmathsetmacro{\edge}{36}

	\node (a) at ({-90 + \edge*0}:\r) {};
	\node (b) at ({-90 + \edge*1}:\r) {};
	\node (c) at ({-90 + \edge*2}:\r) {};
	\node (d) at ({-90 + \edge*3}:\r) {};
	\node (e) at ({-90 + \edge*4}:\r) {};
	\node (f) at ({-90 + \edge*5}:\r) {};

	\draw[edge] (a)--(b);
	\draw[edge] (b)--(c);
	\draw[edge] (c)--(d);
	\draw[dashededge] (d)--(e);
	\draw[edge] (e)--(f);
	\draw[edge] (f)--(a);

	\node[edgelabel] at ({-72 + \edge*0}:\ro) {$\ell_1$};
	\node[edgelabel,right] at ({-72 + \edge*1}:\r) {$\ell_2$};
	\node[edgelabel] at ({-72 + \edge*2}:\ro) {$\ell_3$};
	\node[edgelabel,right] at ({-72 + \edge*3}:\r){$$};
	\node[edgelabel] at ({-72 + \edge*4}:\ro) {$\ell_{n-1}$};
	\node[edgelabel] at (180:\ro /5) {$\ell_n$};

	\node[vertex] at (a) {$\color{white}..$};
	\node[vertex] at (b) {$\color{white}..$};
	\node[vertex] at (c) {$\color{white}..$};
	\node[vertex] at (d) {$\color{white}..$};
	\node[vertex] at (e) {$\color{white}..$};
	\node[vertex] at (f) {$\color{white}..$};
	
	\end{tikzpicture}
\end{center}
\caption{Labeling conventions for general $n$-cycles} \label{ncyclelabels}
\end{figure}
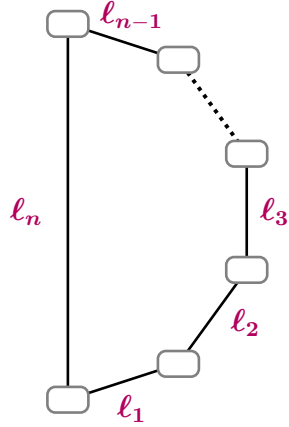

We study the $\Z$-module of splines $R_{G,\alpha}$ in this manuscript.  

\begin{remark} 
\label{unitsandzeros}
We generally assume that the edges of our graphs are not labeled with $0$ or with units.  If the edge $e=v_1v_2$ is labeled with a unit, it does not restrict the splines on the graph since $v_1 \equiv v_2 \mod 1$ is always true.  If an edge $e=v_1v_2$ is labeled zero it tells us that for every spline $\spl{p}{}$ the values $\spl{p}{v_1}=\spl{p}{v_2}$.  

In some of the constructions that follow, we will produce graphs with $0$- or unit-edge-labels.  Given such a graph $G$ we can transform $G$ into a graph $G'$ that does meet our criteria and for which $R_G \cong R_{G'}$.  The transformation merges any two vertices that are joined by an edge labeled $0$ and erases any edge labeled with a unit.  The associated isomorphism of rings of splines is the `forgetful' map that simply omits $\spl{p}{v_2}$ for each edge $e=v_1v_2$ that is labeled $0$, and that is otherwise the identity. 
\end{remark}

A key tool in this paper is the notion of {\em flow-up splines}, which generalizes the concept of a triangular generating set from linear algebra. 

\begin{definition}
Given a graph $G$ with an ordered set of vertices $V = \{v_1, v_2, \ldots, v_n\}$ a flow-up spline for a vertex $v_i$ is a spline $\spl{f^{(i)}}{}$ for which $\spl{f^{(i)}}{v_k}=0$ whenever $k<i$.
\end{definition}

Typically the order is chosen consistently with a direction on the edges of the graph.

It can be very convenient if the entries of flow-up splines have at most one possible nonzero value.  We call these constant flow-up splines, as defined below.

\begin{definition}
\label{constant flow up spline} A {\bf{constant flow up spline}} in $\Z/m\Z$ is a {{flow-up spline}}  $\spl{p}{}$ for which there exists an element $n_i \in \Z/m\Z$ such that $\spl{p_i}{v} \in \{0, n_i\}$ for each $v \in V$.  
\end{definition}

\begin{example}  \label{constantex} We give a set of constant flow-up splines for a graph over $\Z/21\Z$.
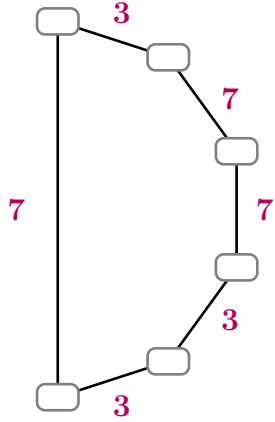
\begin{multicols}{2}
\begin{center}
\begin{tikzpicture}
	\pgfmathsetmacro{\r}{2.5}
	\pgfmathsetmacro{\ro}{2.75}
	\pgfmathsetmacro{\edge}{36}

	\node (a) at ({-90 + \edge*0}:\r) {};
	\node (b) at ({-90 + \edge*1}:\r) {};
	\node (c) at ({-90 + \edge*2}:\r) {};
	\node (d) at ({-90 + \edge*3}:\r) {};
	\node (e) at ({-90 + \edge*4}:\r) {};
	\node (f) at ({-90 + \edge*5}:\r) {};

	\draw[edge] (a)--(b);
	\draw[edge] (b)--(c);
	\draw[edge] (c)--(d);
	\draw[edge] (d)--(e);
	\draw[edge] (e)--(f);
	\draw[edge] (f)--(a);

	\node[edgelabel] at ({-72 + \edge*0}:\ro) {$3$};
	\node[edgelabel,right] at ({-72 + \edge*1}:\r) {$3$};
	\node[edgelabel] at ({-72 + \edge*2}:\ro) {$7$};
	\node[edgelabel,right] at ({-72 + \edge*3}:\r){$7$};
	\node[edgelabel] at ({-72 + \edge*4}:\ro) {$3$};
	\node[edgelabel] at (180:\ro /5) {$7$};

	\node[vertex] at (a) {$\color{white}..$};
	\node[vertex] at (b) {$\color{white}..$};
	\node[vertex] at (c) {$\color{white}..$};
	\node[vertex] at (d) {$\color{white}..$};
	\node[vertex] at (e) {$\color{white}..$};
	\node[vertex] at (f) {$\color{white}..$};
	
	\end{tikzpicture}
\end{center}

\columnbreak

\hspace{-0.5in}
 $ \left\{\left(\begin{array}{c}1\\1\\1\\1\\1\\1\end{array}\right), \left(\begin{array}{c}{0}\\3\\3\\3\\3\\ 0\end{array}\right),\color{black} \left(\begin{array}{c}{0}\\3\\3\\3\\{0}\\0\end{array}\right), \left(\begin{array}{c}{7}\\7\\7\\0\\0\\0\end{array}\right), \left(\begin{array}{c}{7}\\7\\0\\0\\0\\0\end{array}\right) \right\}$
\end{multicols}
\end{example}

\subsection{Flow-up splines generate $R_{G,\alpha}$}

The module $R_{G,\alpha}$ is finite because it is a subset of $(\Z/m\Z)^n$.  The next theorem shows that we can pick flow-up splines to generate the module $R_{G,\alpha}$. Subsequent corollaries use the structure theorem for finite abelian groups to describe $R_{G,\alpha}$ explicitly in terms of flow-up generators. In Sections \ref{generating sets} and \ref{minimal generating sets} we will refine this result to construct flow-up splines that form a minimum generating set for $R_{G,\alpha}$.

\begin{theorem} 
\label{flowupsplinestheorem}
The $\Z$-module $R_{G,\alpha}$ is generated by a collection of flow-up splines.  \end{theorem}

\begin{proof}
We prove the claim by induction.  Our inductive hypothesis is that for some $k$ with $1 \leq k \leq n$ there is a collection of at most $k$ flow-up classes $\{\spl{f^{(i)}}{}: i \leq k\}$ so that each spline $\spl{f}{} \in R_{G,\alpha}$ has integers $c_i \in \Z$ satisfying
\[\spl{f}{v_j}-\sum_{i} c_i \spl{f^{(i)}}{v_j} = 0 \textup{    for all     } v_j \textup{     with     } 1 \leq j \leq k.\]
The base case is $k=1$ for which we take the spline $\spl{f^{(1)}}{}=\spl{1}{}$ to be the identity spline.  If $\spl{f}{} \in R_{G,\alpha}$ then the coefficient $c_1 =\spl{f}{v_1}$ satisfies the condition that $\spl{f}{v_1}- c_1 \spl{f^{(i)}}{v_1}=0$ by construction.

Now assume the claim for $k$.  We show the inductive hypothesis holds for $k+1$ as well.  Consider the $\Z$-submodule $M_{k+1} \subseteq R_{G,\alpha}$ consisting of all splines $\spl{f}{}$ such that $\spl{f}{v_j} = 0$ for all $j$ with $1 \leq j \leq k$.  Now let $\mathcal{I}_{k+1} \subseteq \Z/m\Z$ be the collection
\[ \mathcal{I}_{k+1} = \{ \spl{f}{v_{k+1}} \textup{   for all   } \spl{f}{} \in M_{k+1}\}.\]
As our notation suggests, the set $\mathcal{I}_{k+1}$ is in fact an ideal in $\Z/m\Z$.  Indeed if $\spl{f}{v_{k+1}}$ and $\spl{f'}{v_{k+1}}$ are both in $\mathcal{I}_{k+1}$ then their sum is $\spl{(f+f')}{v_{k+1}}$ which is also in $\mathcal{I}_{k+1}$.  Similarly the integer multiple $c \spl{f}{v_{k+1}}$ is in fact the restriction of the spline $c\spl{f}{}$ to $v_{k+1}$.  Thus the ideal is generated by a single element $\mathcal{I}_{k+1}=\langle a \rangle$.  Let $\spl{f^{(k+1)}}{}$ be any element of $M_{k+1}$ with $\spl{f^{(k+1)}}{v_{k+1}} = a$.  By construction  $\spl{f^{(k+1)}}{}$ is a flow-up spline. (If $a$ is zero we typically take $\spl{f^{(k+1)}}{}$ to be the zero spline.) 

Suppose that $\spl{f}{}$ is an arbitrary spline in $R_{G,\alpha}$.  By the inductive hypothesis there is a linear combination of the flow-up splines $\{\spl{f^{(i)}}{}: i \leq k\}$ so that $\spl{f}{}-\sum_{i} c_i \spl{f^{(i)}}{} \in M_{k+1}$.  By construction of $\spl{f^{(k+1)}}{}$ we know that there is an integer $c_{k+1}$ so that 
\[\spl{f}{v_{k+1}}-\sum_{i} c_i \spl{f^{(i)}}{v_{k+1}} - c_{k+1}\spl{f^{(k+1)}}{v_{k+1}} = 0\]
as desired.  By induction the claim holds.
\end{proof}



\begin{remark}
A collection of flow-up classes that generates the module of splines over an integral domain is in fact a basis, since the flow-up classes are linearly independent by construction.  This may not be true when the base ring is not a domain.  For instance consider the ring $\Z/6\Z$ and let $G$ be the path on two edges whose left edge is labeled $2$ and whose right edge is labeled $3$.  Label the middle vertex $v_1$, the leftmost vertex $v_2$, and the rightmost vertex $v_3$.  Then the flow-up splines from Theorem \ref{flowupsplinestheorem} are $(1,1,1)$, $(0,2,3)$, and $(0,0,3)$ but $3 \cdot (0,2,3) \equiv (0,0,3)$.  So $(0,0,3)$ is generated by $(0,2,3)$.
\end{remark}

\begin{proposition}\label{proposition: basis exists}
The $\Z$-module of splines $R_{G,\alpha}$ satisfies
\[R_{G,\alpha} \cong \bigoplus_i \Z/d_i\Z\]
for positive integers $d_1, \ldots, d_t$ satisfying $d_1 | d_2$, $d_2 | d_3$, $\ldots$, and $d_{t-1} | d_t$.  Moreover the integers $d_i$ are uniquely determined  by $R_{G,\alpha}$.  Finally suppose that for each $i$ the spline $\spl{p_i}{}$ maps to the generator of the summand $\Z/d_i\Z$ under the isomorphism $R_{G,\alpha} \rightarrow \bigoplus_i \Z/d_i\Z$.  Then the set $\{\spl{p_i}{}\}$ is a minimum generating set for $R_{G,\alpha}$.
\end{proposition}

\begin{proof}
This is the invariant factor decomposition of the structure theorem for finite abelian groups (considering the $\Z$-module $R_{G,\alpha}$ as an additive group).  The size of the minimum generating set of a finite abelian group equals the number of factors in the invariant factor decomposition (see, e.g., Dummit and Foote's text \cite[Definition 1 on Page 165, Problem 11 on Page 166]{DumFoo03}).
\end{proof}

\begin{remark}
Example \ref{making generating sets minimal} shows that while the generating set in Example \ref{constantex} is {\em minimal} (namely no subset of those elements generates the same space) it is not {\em minimum} (namely a generating set with the fewest possible elements).  In addition the minimum generating sets we consider are typically not bases in the traditional sense: our modules are generally not free because they are sums of $\Z/m\Z$ for different moduli.  Nonetheless a minimum generating set formed by taking a generator of each factor in the invariant factor decomposition is in some sense canonical and could be called a basis \cite{Dou51a, Dou51b, Dou51c, Dou51d}.  The minimum generating sets we produce in this manuscript are of this form, though we do not refer to them as bases in this manuscript.  The size of a minimum generating set is called the rank of $R_{G,\alpha}$ and denoted $\rk R_{G,\alpha}$. 
\end{remark}

The next lemma combines ideas from the previous two results, showing that a spanning set of constant flow-up splines suffices to characterize the $\Z$-module structure of $R_{G,\alpha}$.

\begin{lemma}
\label{lemma: multiple labels on flow-up classes}
Suppose that $\spl{p_1}{}, \spl{p_2}{}, \ldots, \spl{p_k}{}$ is a set of flow-up generators for $R_{G,\alpha}$and that each  $\spl{p_i}{}$ is a constant flow-up spline, as in Definition \ref{constant flow up spline}.
Then as a $\Z$-module 
\[R_{G,\alpha} \cong \bigoplus_{i=1}^k n_i \Z/m\Z \cong \bigoplus_{i=1}^k \Z/ \frac{m}{\gcd(n_i,m)}\Z.\]
\end{lemma}

\begin{proof}
The second isomorphism follows from the fact that $n_i \Z/m\Z \cong \Z/ \frac{m}{\gcd(n_i,m)}\Z$.  We thus construct an isomorphism $\varphi: R_{G,\alpha} \rightarrow \bigoplus_{i=1}^k n_i \Z/m\Z$ to prove the claim.

Given $\sum c_i \spl{p_i}{} \in R_{G,\alpha}$ define the element 
\[\varphi\left(\sum c_i \spl{p_i}{}\right) = (c_i)_{i=1}^k \in \bigoplus_{i=1}^k n_i \Z/m\Z.\]
This is a $\Z$-module homomorphism so to show that $\varphi$ is well-defined it suffices to show that if $\sum c_i \spl{p_i}{} = \spl{0}{}$ then 
\[\varphi\left(\sum c_i \spl{p_i}{}\right)= (0,0,0,\ldots)\]  
If $\sum c_i \spl{p_i}{} = \spl{0}{}$ then for each $v \in V$ we know $\sum c_i \spl{p_i}{v} \cong 0 \mod m$.  Evaluating at $v_1$ gives $\sum c_i \spl{p_i}{v_1} = c_1 \spl{p_1}{v_1} = c_1 n_1$ because the set $\{\spl{p_i}{}\}$ consists of flow-up splines.  Hence $c_1n_1 \cong 0 \mod m$ and so $c_1 \spl{p_1}{} = \spl{0}{}$.  Assume as the inductive hypothesis that for $i \leq j$ we have $c_in_i \cong 0 \mod m$ and  $c_i \spl{p_i}{} = \spl{0}{}$.  Then evaluate $\sum c_i \spl{p_i}{}$ at $v_j$ to get 
\[\sum c_i \spl{p_i}{v_j} = \sum_{i \leq j} c_i \spl{p_i}{v_j} \cong c_j \spl{p_j}{v_j}\]
by the definition of flow-up splines and then the inductive hypothesis.  Since $\sum c_i \spl{p_i}{} = \spl{0}{}$ we get $c_j \spl{p_j}{v_j} \cong 0 \mod m$.  Since $\spl{p_j}{v_j} = n_j$ we conclude $c_j n_j \cong 0 \mod m$ and hence $c_j \spl{p_j}{}=\spl{0}{}$ as desired.  By induction we conclude that if $\sum c_i \spl{p_i}{} = \spl{0}{}$ then $\varphi\left(\sum c_i \spl{p_i}{}\right)= (0,0,0,\ldots) \in \bigoplus_{i=1}^k n_i \Z/m\Z$.

We now show that the map has a well-defined inverse.  For each $(c_i)_{i=1}^k$ we define
\[\varphi^{-1}((c_i)_{i=1}^k) = \sum c_i \spl{p_i}{}.\]
This is well-defined because if $(c_i)_{i=1}^k \cong (0)_{i=1}^k$ then $c_i n_i \cong 0 \mod m$ for each $i$.  Thus by hypothesis on the splines $\spl{p_i}{}$ we have $c_i \spl{p_i}{} = \spl{0}{}$.  Both $\varphi$ and $\varphi^{-1}$ are $\Z$-module homomorphisms by construction.  We conclude that $\varphi$ is an isomorphism as desired. 
\end{proof}

The next corollary gives one set of conditions under which a generating set of flow-up classes is in fact a minimum generating set.  The proof uses the previous lemma together with the structure theorem for finite abelian groups.

\begin{corollary}\label{corollary: flow-up with one label}
Suppose that $\spl{p_1}{}, \spl{p_2}{}, \ldots, \spl{p_k}{}$ is a set of flow-up generators for $R_{G,\alpha}$ satisfying the following properties:
\begin{itemize}
\item The spline $\spl{p_1}{} = \spl{1}{}$.
\item The splines $\{\spl{p_i}{}: i=2,3,\ldots,k\}$ are constant flow-up splines satisfying $\spl{p_i}{v} \in \{0, a_i\}$ for each $v \in V$ and each $i$.
\item The set $\{a_1=1, a_2, ... , a_k\}$ can be reordered so that $a_{i_1} |  a_{i_2} |  a_{i_3} | ... |  a_{i_k}$.
\end{itemize}
Then $\{  \spl{p_1}{}, \spl{p_2}{}, \ldots, \spl{p_k}{} \}$ forms a minimum generating set for the $\Z$-module $R_{G,\alpha}$.
\end{corollary}

\begin{proof}
Lemma \ref{lemma: multiple labels on flow-up classes} showed that $R_{G,\alpha} \cong  \Z/m\Z \oplus \bigoplus_{j=2}^k (\Z/ \frac{m}{\gcd(a_j,m)}\Z)$.  If $a_j | a_{j'}$ then $gcd(a_j, m) | gcd(a_{j'}, m)$ which in turn implies that $\frac{m}{gcd(a_{j'}, m)} | \frac{m}{gcd(a_j, m)}$.
Up to reordering the factors, this is the invariant factor decomposition of the $\Z$-module $R_{G,\alpha}$.  Thus each minimum generating set of $R_{G,\alpha}$ has exactly $k$ elements.  In particular the splines $\{ \spl{p_i}{}: i=1,2,\ldots,k\}$ form a minimum generating set for $R_{G,\alpha}$.
\end{proof}

\section{Two reductions}

In this section we give two tools to simplify the problem of identifying splines over $G$ mod $m$.   We state the results in terms of the more general ring-theoretic definition and then describe the cases that are most relevant to our applications.  

The first tool describes how splines change after applying homomorphisms to the base ring.  We first describe general results---for instance, that ring surjections  induce spline surjections and that ring injections induce spline injections.  We then give several useful consequences: first that splines over the integers surject onto splines mod $m$; second that splines mod $mm'$ decompose into a direct product of splines mod $m$ and splines mod $m'$ when $m$ and $m'$ are relatively prime. 

The second tool describes how splines over $G$ change after adding a vertex to the graph, and shows that the splines over the larger graph have the structure of a fibration with rank-one fiber.  

\subsection{How changing the base rings affects the ring of splines}

In this section we fix a graph $G$ and describe how homomorphisms of rings induce maps on rings of splines over $G$.  This allows us to conclude several useful things about splines mod $m$: 1) that bases for splines over the integers induce generating sets for splines mod $m$ and 2) that the primary decomposition of the ring $\Z/m\Z$ induces a similar decomposition of the rings of splines.  

We begin by defining a map from ring homomorphisms $\rho: R' \rightarrow R$ to maps on splines $\rho_*: R_{G,\rho^{-1}(\alpha)}' \rightarrow R_{G,\alpha}$.

\begin{definition}
\label{definition: functor}
Consider a ring homomorphism $\rho: R' \rightarrow R$.  Let $G$ be a graph with edge-labeling $\alpha: E \rightarrow \{\textup{ideals in } R\}$.  Denote by $\rho^{-1}(\alpha)$ the edge-labeling that assigns to each edge $e$ the ideal $\rho^{-1}(\alpha(e))$ in the ring $R'$.  

The map $\rho_* : R_{G,\rho^{-1}(\alpha)}' \rightarrow R_{G,\alpha}$ is the restriction of the product map 
\[\rho_*: (R')^{|V|} \rightarrow R^{|V|}\]
 to the rings of splines $\rho_* : R_{G,\rho^{-1}(\alpha)}' \rightarrow R_{G,\alpha}$ namely 
 \[(\rho_*\spl{f}{})_v = \rho(\spl{f}{v})\] 
 for each spline $\spl{f}{} \in R_{G,\rho^{-1}(\alpha)}'$ and each $v \in V$.
\end{definition}

\begin{lemma}
The map $\rho_*: R_{G,\rho^{-1}(\alpha)}' \rightarrow R_{G,\alpha}$ is a well-defined ring homomorphism.
\end{lemma}

\begin{proof}
The map $\rho_*$ is a well-defined ring homomorphism on the ambient rings $\rho_*: (R')^{|V|} \rightarrow R^{|V|}$ by definition.  For each spline $\spl{f}{} \in R_{G,\rho^{-1}(\alpha)}'$ the image $\rho_*\spl{f}{}$ is a spline in $R_{G,\alpha}$ as follows.  For each pair of vertices $u,v \in V$ we have
\[\rho (\spl{f}{u}) - \rho (\spl{f}{v}) \in \rho \left(\rho^{-1}(\alpha(uv))\right)\]
and $\rho \left(\rho^{-1}(\alpha(uv))\right) = \alpha(uv) \subseteq R$ by definition.  
\end{proof}

\begin{corollary}\label{corollary: injection}
Let $G$ be a graph with edge-labeling $\alpha: E \rightarrow \{\textup{ideals in } R\}$. If $\rho: R' \rightarrow R$ is an injection then the map $\rho_*:R_{G,\rho^{-1}(\alpha)}' \rightarrow R_{G,\alpha}$ is an injection.
\end{corollary}

\begin{proof}
If $\rho$ is an injection then the map $\rho_*: (R')^{|V|} \rightarrow R^{|V|}$ is an injection by definition, so the restriction $\rho_*:R_{G,\rho^{-1}(\alpha)}' \rightarrow R_{G,\alpha}$ is also an injection.
\end{proof}

In addition surjections of rings induce surjections of splines.  

\begin{proposition}
\label{surjectionproposition}
Let $G$ be a graph with edge-labeling $\alpha: E \rightarrow \{\textup{ideals in } R\}$.  If $\rho: R' \rightarrow R$ is a surjection then $\rho_* : R_{G,\rho^{-1}(\alpha)}' \rightarrow R_{G,\alpha}$ is a surjection.  
\end{proposition}

\begin{proof}
Let $\spl{f}{} \in R_{G,\alpha}$.  For each $v \in V$ choose an element $n_v \in \rho^{-1}(\spl{f}{v})$.  Define the element $\spl{f'}{} \in (R')^{|V|}$ by $\spl{f'}{v} = n_v$ for each vertex $v \in V$.  We show that $\spl{f'}{}$ is actually a spline in $R_{G,\rho^{-1}(\alpha)}'$.  Suppose that $u,v$ are adjacent vertices.  Then $\spl{f'}{u}-\spl{f'}{v} = n_u-n_v$ by construction.  Consider the image $\rho(n_u-n_v)$ in $R$. By our choice of $n_u$ and $n_v$ we know $\rho(n_u)-\rho(n_v) = \spl{f}{u}-\spl{f}{v}$ and so $\rho(n_u-n_v) \in \alpha(uv)$.  Thus $n_u-n_v \in \rho^{-1}(\alpha(uv))$ and hence $\spl{f'}{} \in R_{G,\rho^{-1}(\alpha)}'$. This proves the claim.
\end{proof}

We give two applications in the context of splines mod $m$.  The first shows that we can infer information about splines mod $m$ from splines over the integers.  It follows immediately from Proposition \ref{surjectionproposition}.

\begin{corollary}\label{corollary: integers subject onto mod $m$}
Let $G$ be a graph with edge-labeling $\alpha$ in $\Z/m\Z$.  Let $R' = \Z$ denote the ring of integers and $R = \Z/m\Z$ denote the ring of integers mod $m$.  Then the natural surjection $\rho: \Z \rightarrow \Z/m\Z$ induces a surjection  $\rho_*: R_{G,\rho^{-1}(\alpha)}' \rightarrow R_{G,\alpha}$. In particular if $\{\spl{b_1}{}, \spl{b_2}{}, \ldots, \spl{b_n}{}\}$ is a basis for the $\Z$-module of splines $R_{G, \rho^{-1}(\alpha)}'$ over the integers then $\{\rho_* \spl{b_1}{}, \rho_* \spl{b_2}{}, \ldots, \rho_* \spl{b_n}{}\}$ spans the $\Z$-module of splines $R_{G,\alpha}'$ over $\Z/m\Z$.
\end{corollary}

For instance we  use Corollary \ref{corollary: integers subject onto mod $m$} to construct a three-cycle whose splines have rank one over $\Z/m\Z$ where $m$ has at least three prime factors.

 \begin{example}
  \label{new find} If $(C_3, L)$ is three cycle mod $pqr$ for some prime numbers $p,q,r$ and if  $L= \{pq,qr,rp\}$ then the trivial spline $\spl{1}{}$ generates all splines on $(C_3, L)$.  (The labeling in the diagram is fully general.)    By \cite{HMR} or \cite{BHKR} an integer basis for this edge-labeling is $\{[1,1,1]^T, [0,pqr,pqr]^T, [0,0,pqr]^T\}$.  Thus all splines mod $pqr$ are trivial.

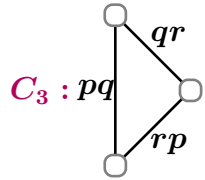
\begin{multicols}{2}

\begin{center}

\begin{tikzpicture}

	\begin{scope}
		\pgfmathsetmacro{\r}{1}

		\draw[edge][color=black] (-90:\r)--(0:\r);
		\draw[edge][color=black] (0:\r)--(90:\r);
		\draw[edge][color=black] (90:\r)--(-90:\r);		

		\node[edgelabel] at (-45:\r){\color{black}{$rp$}};
		\node[edgelabel] at (45:\r) {\color{black}{$qr$}};
		\node[edgelabel] at (180:\r /4) {\color{black}{$pq$}};
		
		\node[vertex] at (-90:\r) {$\scriptstyle{}$};
		\node[vertex]at (0:\r) {$\scriptstyle{}$};
		\node[vertex]at (90:\r) {$\scriptstyle{}$};

		\node[edgelabel] at (-1,0) {$C_3:$};
	\end{scope}

\end{tikzpicture}
\end{center}

\columnbreak

   \begin{center}  $\mathbb{B}(R_G) =   \left\{\left(\begin{array}{c}1\\1\\1\end{array}\right)\right\}$ 
   \end{center}

\end{multicols}
\end{example}

The second example of this reduction exploits the isomorphism $\Z/m\Z \cong \Z/m'\Z \oplus \Z/m''\Z$ when $m'm''=m$ and $\gcd(m',m'')=1$ to obtain a structure theorem for splines mod $m$.  This is similar to the technique of  localizing at prime ideals for polynomial rings, which has been used in work on algebraic splines \cite{BilRos91, DiP14, Yuz92}.

\begin{theorem}\label{structure theorem}
Let $R = \Z/m\Z$ with $m = m'm''$ and $gcd(m',m'')=1$.  Denote the ring $R'=\Z/m'\Z$ and the standard quotient map $\rho': \Z/m\Z \rightarrow \Z/m'\Z$, respectively $R''=\Z/m''\Z$ and $\rho'': \Z/m\Z \rightarrow \Z/m''\Z$. Let $G$ be a graph with edge-labeling $\alpha$ over the integers, and let $\alpha'$ (respectively $\alpha''$) denote the edge-labeling function that sends each edge $uv$ to the ideal $\alpha'(uv)=\rho'(\alpha(uv))$.

Then 
\[R_{G,\alpha} \cong R'_{G,\alpha'} \oplus R''_{G,\alpha''}.\]
\end{theorem}

\begin{proof}
 By definition of the direct sum we have $R_{G,\alpha' \oplus \alpha''} \cong  R_{G,\alpha'} \oplus R_{G,\alpha''}$. By construction the edge-labelings $\alpha'$ and $\alpha''$ satisfy $(\rho')^{-1}(\alpha')=\alpha$ and $(\rho'')^{-1}(\alpha'') = \alpha$. Thus the Chinese Remainder Theorem guarantees that $(\rho' \oplus \rho'')^{-1}(\alpha'(uv) \oplus \alpha''(uv)) = \alpha(uv)$ for each edge $uv$. We conclude that the isomorphism $\rho' \oplus \rho'': R \rightarrow R' \oplus R''$ induces a map $(\rho' \oplus \rho'')_* : R_{G,\alpha} \rightarrow R'_{G,\alpha'} \oplus R''_{G,\alpha''}$.  By Corollary \ref{corollary: injection} we know $(\rho' \oplus \rho'')_*$ is an injection and by Proposition \ref{surjectionproposition} we know $(\rho' \oplus \rho'')_*$ is a surjection.  It follows that $R_{G,\alpha} \cong R'_{G,\alpha'} \oplus R''_{G,\alpha''}$ as desired.
\end{proof}

The next corollary describes how to use Theorem \ref{structure theorem} to identify minimum generating sets of and the rank of the $\Z$-module of splines mod $m$.

\begin{corollary}\label{corollary: invariant factor decomposition}
Suppose that $m=p_1^{e_1} p_2^{e_2} \cdots p_k^{e_k}$ is the primary decomposition of $m$.  For each $i=1,2,\ldots,k$ denote the ring $R_i=\Z/p_i^{e_i}\Z$ and the standard quotient map $\rho_i: \Z/m\Z \rightarrow \Z/p_i^{e_i}\Z$.  Let $G$ be a graph with edge-labeling $\alpha$ over the integers, and for each $i$ let $\alpha_i$  denote the edge-labeling function that sends each edge $uv$ to the ideal $\alpha_i(uv)=\rho_i(\alpha(uv))$.

Then 
\[R_{G,\alpha} \cong \bigoplus_{i=1}^k  (R_i)_{G,\alpha_i} \]
and the rank of $R_{G,\alpha}$ is the maximum $\max \left\{\textup{rk} (R_i)_{G,\alpha_i} \textup{  for  } i=1,2,\ldots,k\right\}$.
\end{corollary}

\begin{proof}
The first claim follows from Theorem \ref{structure theorem} and a small induction.  For each $i$ denote a minimum generating set for $(R_i)_{G,\alpha_i}$ by $\{\spl{b_{p_i}^j}{}: j=1,\ldots,d_i\}$.  Let $d$ be the maximum of the $d_i$ and take $\spl{b_{p_i}^j}{}$ to be $\spl{0}{}$ for each $j$ with $d_i < j \leq d$. We construct a minimum generating set for $R_{G,\alpha}$ by finding the preimage in $R_{G,\alpha}$ of each tuple $\left(\spl{b_{p_1}^j}{}, \spl{b_{p_2}^j}{}, \ldots, \spl{b_{p_k}^j}{}\right)$ for $j=1,\ldots, d$. The set generates $R_{G,\alpha}$ because its projection generates the isomorphic ring $\bigoplus_{i=1}^k  (R_i)_{G,\alpha_i}$.  It is a minimum generating set because $d$ elements are needed to generate the $\spl{b_{p_i}^j}{}$ in one of the factors in the isomorphic ring $\bigoplus_{i=1}^k  (R_i)_{G,\alpha_i}$.   In particular the rank of $R_{G,\alpha}$ is $d$ as desired.
\end{proof}

\begin{example} We can use the elementary divisors of $\Z/m\Z$ to determine the invariant factor decomposition of the ring of splines $R_{G,\alpha}$ mod $m$ as a $\Z$-module.  
 
 Let $G$ be the following edge-labeled triangle and consider the splines $R_{G,\alpha}$ over $\Z/36\Z$.
The elementary divisors of $\Z/36\Z$ are $\Z/4\Z$ and $\Z/9\Z$.  We project to these elementary divisors in the graphs $H_1$ and $H_2$. The graph $H_1$ is identical to $G$ except that every edge label from $G$ is now considered mod $4$.  Similarly in $H_2$ we take the edges from $G$ mod $9$.

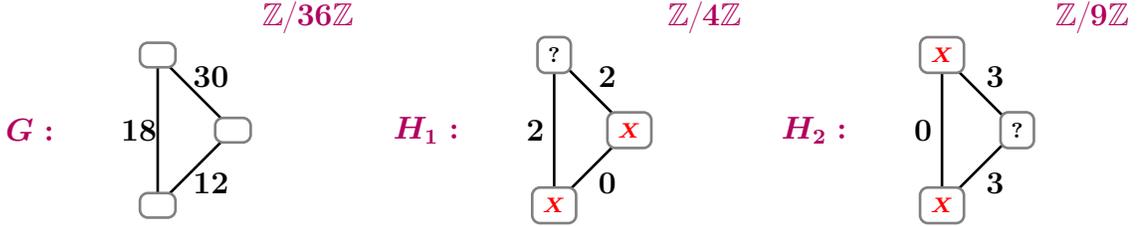
\begin{multicols}{3}
\begin{center}

\begin{tikzpicture}

	\begin{scope}
		\pgfmathsetmacro{\r}{1}

		\draw[edge][color=black] (-90:\r)--(0:\r);
		\draw[edge][color=black] (0:\r)--(90:\r);
		\draw[edge][color=black] (90:\r)--(-90:\r);		

		\node[edgelabel] at (-45:\r){\color{black}{$12$}};
		\node[edgelabel] at (45:\r) {\color{black}{$30$}};
		\node[edgelabel] at (180:\r /4) {\color{black}{$18$}};
		
		\node[vertex] at (-90:\r) {$\scriptstyle{\color{white}..}$};
		\node[vertex]at (0:\r) {$\scriptstyle{\color{white} ..}$};
		\node[vertex]at (90:\r) {$\scriptstyle{\color{white} ..}$};

		\node[edgelabel] at (2,1.5) {$ \Z/36\Z$};
	\node[edgelabel] at (-1.7,0) {$G:$};
	\end{scope}

\end{tikzpicture}

\end{center}

\columnbreak

\begin{center}

\begin{tikzpicture}

	\begin{scope}
		\pgfmathsetmacro{\r}{1}

		\draw[edge][color=black] (-90:\r)--(0:\r);
		\draw[edge][color=black] (0:\r)--(90:\r);
		\draw[edge][color=black] (90:\r)--(-90:\r);		

		\node[edgelabel] at (-45:\r){\color{black}{$0$}};
		\node[edgelabel] at (45:\r) {\color{black}{$2$}};
		\node[edgelabel] at (180:\r /4) {\color{black}{$2$}};
		
		\node[vertex] at (-90:\r) {$\scriptstyle{\color{red}X}$};
		\node[vertex]at (0:\r) {$\scriptstyle{\color{red}X}$};
		\node[vertex]at (90:\r) {$\scriptstyle{\color{black} ?}$};

		\node[edgelabel] at (2,1.5) {$ \Z/4\Z$};
		\node[edgelabel] at (-1.7,0) {$H_1:$};
	\end{scope}

\end{tikzpicture}

\end{center}

\columnbreak

\begin{center}

\begin{tikzpicture}

	\begin{scope}
		\pgfmathsetmacro{\r}{1}

		\draw[edge][color=black] (-90:\r)--(0:\r);
		\draw[edge][color=black] (0:\r)--(90:\r);
		\draw[edge][color=black] (90:\r)--(-90:\r);		

		\node[edgelabel] at (-45:\r){\color{black}{$3$}};
		\node[edgelabel] at (45:\r) {\color{black}{$3$}};
		\node[edgelabel] at (180:\r /4) {\color{black}{$0$}};
		
		\node[vertex] at (-90:\r) {$\scriptstyle{\color{red}X}$};
		\node[vertex]at (0:\r) {$\scriptstyle{\color{black} ?}$};
		\node[vertex]at (90:\r) {$\scriptstyle{\color{red}X}$};

		\node[edgelabel] at (2,1.5) {$ \Z/9\Z$};
		\node[edgelabel] at (-1.7,0) {$H_2:$};

	\end{scope}

\end{tikzpicture}

\end{center}

\end{multicols}

Next we determine a minimum generating set for $H_1$ and $H_2$.  When two vertices $v_1, v_2$ are joined by an edge labeled $0$ then since $v_1 \equiv v_2 \mod 0$ we conclude $v_1 = v_2$.  This means both $H_1$ and $H_2$ are  effectively single edges.  Thus the minimum generating sets for $H_1$ and $H_2$ are 

\begin{center}
   $\mathbb{M}({H_1}) =  \left\{\left(\begin{array}{c}1\\1\\1\end{array}\right), \left(\begin{array}{c}2\\0\\0\end{array}\right)\right\}$ 
\hspace{0.25in}
   $\mathbb{M}(H_2) = \left\{\left(\begin{array}{c}1\\1\\1\end{array}\right), \left(\begin{array}{c}0\\3\\0\end{array}\right)\right\}$ 

\end{center}

To find the minimum generating set for splines over $G$ mod $6$ we use the same strategy as we would for finding elementary divisors.  In particular we find the unique spline mod $36$ that projects to $[1,1,1]^T \mod 4$ and $[1,1,1]^T \mod 9$ and the unique spline mod $6$ that projects to $[2,0,0]^T \mod 4$ and $[0,3,0]^T \mod 9$. We obtain  

\begin{center}
   $\mathbb{M}({G}) =  \left\{\left(\begin{array}{c}1\\1\\1\end{array}\right), \left(\begin{array}{c}18\\12\\0\end{array}\right)\right\}$ 
\end{center}

\end{example}

\begin{example} \label{4cycleremark}  Suppose $p$ and $q$ are distinct primes.  Label the edges of the complete graph $K_4$ as indicated below. Projecting to splines mod $p$ the edges $x_4x_3$, $x_3x_1$, and $x_1x_2$ all become zero so the vertices $x_2, x_1, x_3, x_4$ all agree mod $p$.  Projecting to splines mod $q$ the edges $x_1x_4$, $x_4x_2$, and $x_2x_3$ all become zero so the vertices $x_1,x_4,x_2,x_3$ all agree mod $q$.  Thus all splines on this graph mod $p$ are trivial and similarly mod $q$.  From the isomorphism $\Z/pq\Z \rightarrow \Z/p\Z \oplus \Z/q\Z$ we conclude that all splines mod $pq$ on this edge-labeled graph are trivial splines. 

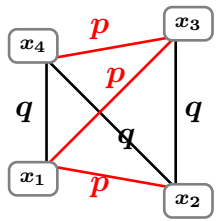
\begin{center}
\begin{tikzpicture}

	\begin{scope}
		\pgfmathsetmacro{\r}{1}

		\draw[edge][color=black] (-45:\r/.7)--(45:\r/.7);
		\draw[edge][color=black] (225:\r)--(135:\r);
		\draw[edge][color=red] (225:\r)--(-45:\r/.7);
		\draw[edge][color=black] (135:\r)--(-45:\r/.7);		
		\draw[edge][color=red] (135:\r)--(45:\r/.7);
		\draw[edge][color=red] (225:\r/.8)--(45:\r/.6);

		\node[edgelabel] at (90:\r/.9){\color{red}{$p$}};
		\node[edgelabel] at (65:\r/2) {\color{red}{$p$}};
		\node[edgelabel] at (-90:\r) {\color{red}{$p$}};
		\node[edgelabel] at (0:\r/.8){\color{black}{$q$}};
		\node[edgelabel] at (180:\r) {\color{black}{$q$}};
		\node[edgelabel] at (-45:\r/2) {\color{black}{$q$}};

		\node[vertex] at (-45:\r/.6) {$\scriptstyle{x_2}$};
		\node[vertex]at (45:\r/.6) {$\scriptstyle{x_3}$};
		\node[vertex]at (135:\r/.8) {$\scriptstyle{x_4}$};
		\node[vertex]at (225:\r/.8) {$\scriptstyle{x_1}$};

	\end{scope}

\end{tikzpicture}

\end{center}

\end{example}

\subsection{Adding one vertex to the graph}  Suppose that the graph $G^+$ is obtained from the graph $G$ by adding a single vertex and some number of edges. In this section we investigate how the rings $R_{G,\alpha}$ and $R_{{G^{+}},{\alpha^{+}}}$ can differ.  The key tool is the forgetful map from splines on $G^+$ to splines on $G$, which is well-defined by the following result \cite[Proposition 2.8]{GPT}.

\begin{proposition}[Gilbert--Polster--Tymoczko] \label{addingvertexprop}
If $G$ is a subgraph of $G'$ then every spline in $R_{G'}$ restricts to a spline in $R_{G}$ via the forgetful map $\pi: R_{G'} \rightarrow R_{G}$ that omits the vertices in $V(G')-V(G)$ and their incident edges.
\end{proposition}

We use  Proposition \ref{addingvertexprop} repeatedly to guarantee that when we add a vertex to $G$ we create no new splines   on $G$ itself. Loosely speaking we will show that each spline on the expanded graph $G^+$ consists of the sum of a spline coming from $G$ and a spline supported exactly on the new vertex. More formally we have a sequence of maps as follows.

\begin{lemma}\label{lemma: adding one vertex}
Let $R$ be a principal ideal ring.  Let $G$ be an edge-labeled graph and let $G^+$ be a graph obtained from $G$ by adding a vertex $v$ plus some edges between $v$ and vertices in $G$.   Let $N$ be the least common multiple of the labels on edges incident to $v$.  Let $M_v$ be the collection of splines in $R_{G^+}$ that are supported exactly on $v$.  Define the map $\pi: R_{G^+} \rightarrow R_G$ by forgetting $v$ and its incident edges.  Then $\ker \pi \cong M_v \cong N R$.
\end{lemma}

\begin{proof}
First we show $M_v \cong N R$.  If $\spl{f}{} \in R_{G^+}$ and $\spl{f}{u}=0$ for all $u$ adjacent to $v$ then $\spl{f}{v}$ is a multiple of the label on each edge incident to $v$.  In other words $\spl{f}{v}$ is a multiple of $N$.  Moreover $\spl{f}{v}$ can be any multiple of $N$ because 
\[\spl{f}{v}=rN \textup{  and  } \spl{f}{u}=0 \textup{  for all  } u \neq v\]
satisfies all the edge conditions for each element $r \in R$.

Now we show that $M_v = \ker \pi$.  Indeed if $\spl{f}{} \in M_v$ then by definition of $\pi$ we know that $\pi(\spl{f}{})  = \spl{0}{}$.  Similarly if $\pi(\spl{f}{}) = \spl{0}{}$ then $\spl{f}{u} = 0$ for all vertices $u \neq v$.  That means $\spl{f}{} \in M_v$ and proves the claim.
\end{proof}

\begin{remark}
We typically use this lemma together with the splitting principle, which states that if 
\[0 \rightarrow A \hookrightarrow M \stackrel{\pi}{\rightarrow} B \rightarrow 0\]
is an exact sequence of $R$-modules and $\varphi: B \rightarrow M$ is a map such that $\pi \circ \varphi: B \rightarrow B$ is the identity then $M \cong A \oplus B$ as $R$-modules.  We will generally consider all of our modules (both spline modules $R_G$ and the submodules $M_v$) as $\Z$-modules.
\end{remark}

For some maps $\pi$ the ring of splines $R_{G^+}$ is particularly easy to identify.

\begin{corollary}
 Let $m$ be an integer.  Let $G$ be an edge-labeled graph and let $G^+$ be a graph obtained from $G$ by adding a vertex $v$ plus some edges between $v$ and vertices in $G$.  Suppose two of the edges incident to $v$ are labeled $n_1$ and $n_2$ where $\lcm (n_1,n_2)=m$.  Then $R_{G^+} \cong \textup{Im } \pi$.
\end{corollary}

\begin{proof}
In this case $\ker \pi$ is isomorphic to $\{ \spl{0}{} \} \subseteq \Z/m\Z$ so the claim follows.
\end{proof}

The map $\pi$ may not be surjective, not even when the base ring is the integers.

\begin{example}
Consider $G$ and $G^{+}$ below.  Handschy, Melnick, and Reinders proved that the splines $[1,1,1]$, $[0,6,6]$, and $[0,0,6]$ form a basis for $R_{G^+}$ as part of a larger result about integer splines on cycles \cite{HMR}.  But the basis for $R_G$ consists of $[1,1]$ and $[0,2]$ \cite[Theorem 4.1]{GPT}.  In particular the spline $[0,2]$ is not in the image of $R_{G^+} \stackrel{\pi_1}{\rightarrow} R_G$. 

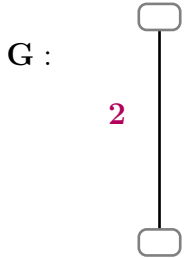
\begin{multicols}{2}

\begin{tikzpicture}
	\pgfmathsetmacro{\r}{1.5}
	\pgfmathsetmacro{\ro}{2.75}
	\pgfmathsetmacro{\edge}{36}

	\node (a) at ({90 + \edge*0}:\r) {};
	\node (b) at ({-90 + \edge*0}:\r) {};
	\node(label) at  ({-30}:-2) {};

	\draw[edge] (a)--(b);

	\node[edgelabel] at ({160}:.60) {$2$};
	
	\node[vertex] at (a) {$\color{white}..$};
	\node[vertex] at (b) {$\color{white}..$};
	\node[label] at (label){${\bf{G:}}$};

	\end{tikzpicture}

\columnbreak

\begin{center}

A basis for $R_G$:

 $ \left\{\left(\begin{array}{c}1\\1\\\end{array}\right), \color{red} \left(\begin{array}{c}{2}\\ 0\end{array}\right) \color{black} \right\}$

\end{center}

\end{multicols}

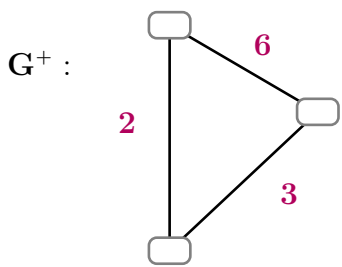
\begin{multicols}{2}

\begin{tikzpicture}
	\pgfmathsetmacro{\r}{1.5}
	\pgfmathsetmacro{\ro}{2.75}
	\pgfmathsetmacro{\edge}{36}

	\node (a) at ({90 + \edge*0}:\r) {};
	\node (b) at ({-90 + \edge*0}:\r) {};
	\node(c) at ({10}:2.){};
	\node(label) at  ({-30}:-2) {};

	\draw[edge] (a)--(b);
	\draw[edge] (c)--(b);
	\draw[edge] (c)--(a);

	\node[edgelabel] at ({160}:.60) {$2$};
	\node[edgelabel] at ({45}:1.75) {$6$};
	\node[edgelabel] at ({-25}:1.75) {$3$};
	
	\node[vertex] at (a) {$\color{white}..$};
	\node[vertex] at (b) {$\color{white}..$};
	\node[vertex] at (c) {$\color{white}..$};
	\node[label] at (label){${\bf{G^{+}:}}$};

	\end{tikzpicture}

\columnbreak

\begin{center}

A basis  for $R_{G^{+}}$:

 $ \left\{\left(\begin{array}{c}1\\1\\1\end{array}\right), \left(\begin{array}{c}{6}\\6\\ 0\end{array}\right),\color{black} \left(\begin{array}{c}{6}\\{0}\\0\end{array}\right) \right\}$

\end{center}

\end{multicols}

\end{example}

\section{Splines mod $m$ can have almost any rank between $1$ and $n$}

The rank of the module of splines on a graph with $n$ vertices is tightly constrained over the integers: the rank must be $n$.  This is no longer true mod $m$ as we have already seen in several examples.  In this section we show in fact the rank of a module of splines on a graph with $n$ vertices mod $m$ can be essentially any integer between $1$ and $n$.  (The only exceptions are when $m$ has few prime factors and $n$ is very small.)  Our proof has three main steps: we bound the rank of the module of splines by the number of vertices; we then show that any integer between $2$ and $n$ can be achieved; and we finish by showing when we can construct graphs with arbitrarily large number of vertices whose ring of splines mod $m$ are trivial.

We start by proving that the maximum rank of a module of splines mod $m$ is the number of vertices in the graph.  We prove the result for the base ring $\Z/m\Z$ by using the quotient map $\Z \rightarrow \Z/m\Z$.  For this reason we include the result for the integers, too.

\begin{theorem} \label{theorem: max rank is n}
Suppose that $G$ is a graph with $n$ vertices.  Both over the integers and over the ring $\Z/m\Z$ the maximum rank of a ring of splines $R_{G}$ is $n$.
\end{theorem}

\begin{proof} 
We prove the case of the integers first using induction on the number of vertices $n$.  The base case is when $n=1$ in which case $R_{G,\alpha}$ has rank $1$ by (trivial) definition.  The inductive hypothesis is that if $G$ has $n$ vertices then $R_G$ has rank at most $n$.  Now suppose that $G^+$ is obtained from $G$ by adding one vertex $v$ and some number of edges and define $\pi$ as in Lemma \ref{lemma: adding one vertex}.  Then $R_{G^+} \cong \textup{Im }\pi \oplus M_v$.  Since $\textup{Im }\pi$ is a submodule of the free module $R_G$ over the  principal ideal domain $\Z$ it is free of rank at most $\textup{rk }R_G$ (see e.g. \cite{McN}).
Lemma \ref{lemma: adding one vertex} showed that $M_v$ has rank at most one.  Thus $R_{G^+}$ has rank at most $n+1$ as desired.

Now let $G$ be a graph with $n$ vertices and let $\alpha$ be an edge-labeling over $R=\Z/m\Z$.  Denote the quotient map by $\rho: \Z \rightarrow \Z/m\Z$ and let $\alpha'$ be the  edge-labeling over the integers with $\rho(\alpha'(uv))= \alpha(uv)$ for each edge $uv$ in $G$.  Denote the integers by $R'=\Z$.  Corollary \ref{corollary: integers subject onto mod $m$} showed that each set of generators for the ring of splines $R_{G,\alpha'}'$ over the integers surjects onto a set of generators for $R_{G,\alpha}$ over the integers mod $m$.  Thus the rank of $R_{G,\alpha}$ is also at most $n$.
\end{proof}

The next result constructs graphs with arbitrarily large number of vertices $n$ and with any rank between $2$ and $n$.  Our proof builds a graph with the desired rank vertex-by-vertex using constant flow-up splines; by controlling the label of one edge incident to the new vertex, we can determine whether the rank of the larger graph increases by one or stays constant.  Note that the theorem assumes the modulus is not a power of a prime; Corollary \ref{corollary: cycles and prime powers} describes the module of splines when $m$ is a power of a prime in the special case of cycles.  

\begin{theorem}
If $m$ has at least two distinct prime factors then for each $n\geq 2$ and each $i$ with $2 \leq i \leq n$ there exists an edge-labeled graph $G$ on $n$ vertices with $\rk R_G =i$.
\end{theorem}

\begin{proof}
Choose two relatively prime factors $n_1, n_2$ with $n_1n_2=m$.  This is possible because $m$ has at least two distinct prime factors.  We prove the theorem by induction.  Our inductive hypothesis is that for $n \geq 2$ we can construct an edge-labeled graph $G_i$ on $n$-vertices for each $i$ with $2 \leq i \leq n$ such that
\begin{itemize}
\item The module $R_{G_i}$ has a minimum generating set consisting of the trivial spline $[1,1,\ldots,1]$ and $i-1$ other flow-up splines each of whose entries are either $n_1$ or $0$.
\end{itemize} 

The base case is when $n=2$.  In this case the graph with a single edge labeled $\langle n_1 \rangle$ has minimum generating set $[1,1]$ and $[0,n_1]$ satisfying the inductive hypothesis.

Now assume the graph $G_i$ satisfies the inductive hypothesis for $n$.  We will add a vertex and several edges to $G_i$ in two ways, one of which increases the rank of $R_{G_i}$ by exactly one and one of which preserves the rank of $R_{G_i}$.

First create the graph $G_i'$ by adding a vertex $v$ together with at least two edges to vertices in $G_i$.  Label all new edges in $G_i'$ with $\langle n_1 \rangle$.  Next create the graph $G_i''$ by adding a vertex $v$ together with at least two edges to vertices in $G_i$.  Label one of the edges in $G_i''$ with $\langle n_2 \rangle$ and label the rest $\langle n_1 \rangle$.  Denote the edge labeled $\langle n_2 \rangle$ in $G_i''$ by $v'v$.

We show that for both $G_i'$ and $G_i''$ the maps $\pi': R_{G_i'} \rightarrow R_{G_i}$ and $\pi'': R_{G_i''} \rightarrow R_{G_i}$ from Lemma \ref{lemma: adding one vertex} are surjective. Indeed let $\spl{b}{}$ be a spline from the minimum generating set for $R_{G_i}$ given in the inductive hypothesis.  Then we extend $\spl{b}{}$ to a spline 
\begin{itemize}
\item $\spl{b}{}' \in R_{G_i'}$ by defining $\spl{b}{v}'=0$ and 
\item $\spl{b}{}' \in R_{G_i''}$ by defining $\spl{b}{v}' = \spl{b}{v'}.$
\end{itemize}
Whether we set $\spl{b}{v}' = n_1$ or $\spl{b}{v}' = 0$ the result satisfies all edge conditions in $G_i'$ and all edge conditions except perhaps for edge $v'v$ in $G_i''$.   By construction $\spl{b}{v}'-\spl{b}{v'}' = 0$ and so in both cases $\spl{b}{}'$ is a spline.   

Recall that the collection of splines $M_v$ from Lemma \ref{lemma: adding one vertex} consists of all splines that are zero at each vertex of $G_i$.  For $G_i'$ we have the spline $\spl{b}{}^v \in M_v$ defined by $\spl{b}{v}^v = n_1$ and $\spl{b}{u}^v=0$ for all other vertices $u$.  Moreover $\spl{b}{}^v$ generates $M_v \subseteq R_{G_i'}$ by Lemma \ref{lemma: adding one vertex}, proving that $M_v \cong n_1 R$.  For $G_i''$ we have $M_v \cong  n_1n_2 R = \{ 0 \}$ also by Lemma \ref{lemma: adding one vertex}. 

In both cases Lemma \ref{lemma: adding one vertex} together with the splitting principle guarantee a generating set.  For $R_{G_i'}$ we have the generating set 
\[\{\spl{b}{}' : \spl{b}{} \textup{ is in the inductive minimum generating set}\} \cup \{\spl{b}{}^v\}\] 
while for $R_{G_i''}$ the splines $\{\spl{b}{}' : \spl{b}{} \textup{ is in the inductive minimum generating set}\}$ alone span $R_{G_i''}$.  Corollary \ref{corollary: flow-up with one label} guarantees that both generating sets are in fact {\em minimum} generating sets.  By inspection both minimum generating sets satisfy the inductive hypothesis.  So by induction the claim is proven.
\end{proof}

Finally we show that for most moduli and for most $n$ we can construct a graph with only trivial splines.  The proof is similar to the previous theorem.

\begin{theorem}
\label{edge and set size theorem}
If 
\begin{itemize}
\item $m$ has at least three distinct prime factors and $n \geq 3$, or
\item $m$ has at least two distinct prime factors and $n \geq 4$
\end{itemize}
then there exists an edge-labeled graph $G$ on $n$ vertices with $\rk R_G = 1$.
\end{theorem}

\begin{proof}
Induct on the number $n$ of vertices.  The base case is either Example \ref{new find} or Example \ref{4cycleremark}.  Assume the claim holds for a given modulus $m$ and number of vertices $n$.  Let $G$ be the edge-labeled graph given by the inductive hypothesis.

The module of generalized splines $R_G$ always contains the trivial splines, so if $\rk R_G = 1$ then $R_G$ consists solely of the trivial splines.  Choose two relatively prime integers $n_1, n_2$ so that $n_1n_2=m$.  Construct a graph $G'$ from $G$ by adding one new vertex $v$ and at least two new edges between $v$ and vertices in $G$.  Label the edges from $v$ so that at least one is labeled $\langle n_1 \rangle$ and at least one is labeled $\langle n_2 \rangle$.  

The projection $\pi: R_{G'} \rightarrow R_G$ from Lemma \ref{lemma: adding one vertex} is a surjection in this case since every trivial spline on $G$ can be extended to a trivial spline on $G'$.  Lemma \ref{lemma: adding one vertex} and the splitting principle together imply $R_{G'} \cong R_G \oplus M_v$.  By construction $M_v \cong (n_1n_2)\Z/m\Z$ which is just zero.   Thus $R_{G'} \cong R_G$.  We conclude that $R_{G'}$ also consists solely of the trivial splines.  Every ring of splines contains all multiples of the identity spline so $\rk R_{G'} \geq 1$.  We conclude $\rk R_{G'} = 1$ as desired.  
\end{proof}

\begin{remark}
These are sharp results, in the sense that every graph with two vertices has rank at least two, and every graph with three vertices and exactly two distinct prime factors in its modulus has rank at least two. Indeed every graph on two vertices has rank $2$ by inspection or by \cite[Theorem 4.1]{GPT}.  

To see the case of 3-cycles for which $m$ has exactly two distinct prime factors, suppose the edge-labels are $p^{i_1}q^{j_1}$, $p^{i_2}q^{j_2}$, and $p^{i_3}q^{j_3}$.  The least common multiple of $p^{i_1}q^{j_1}$ and  $p^{i_2}q^{j_2}$ is $p^{\max i_1,i_2}q^{\max j_1,j_2}$.  Suppose this equals $m$.  No edge is labeled with $m$ so assume without loss of generality that $i_1<i_2$ and $j_1 > j_2$.  The third edge must have either $i_3<i_2$ or $j_3<j_1$.  Assume without loss of generality that $i_3<i_2$.  Then the least common multiple of  $p^{i_1}q^{j_1}$ and  $p^{i_3}q^{j_3}$ is $d=p^{\max i_1,i_3}q^{\max j_1,j_3}$ which is strictly less than $m$ by construction.  Thus the spline with $d$ at the vertex incident to the first and the third edge and zero elsewhere is a nontrivial spline on the graph.
\end{remark}

\begin{question}
What are the possible ranks for splines mod $p^k$?
\end{question}

\section{Characterizing Rings of Splines in Particular Cases}\label{section: classification}

In this section we present classifications of splines over specific moduli.  In some instances we are able to characterize all splines over arbitrary graphs for a given modulus, while in other cases our current characterizations are restricted to cycles.  Many of the results that follow naturally provide a minimum generating set (for instance Theorem \ref{singleedge}).  In some cases we obtain a generating set that is not necessarily minimum.  Thus   Section {\ref{minimal generating sets}  gives an algorithm to transform the generating sets in Section \ref{generating sets} into minimum generating sets.

\subsection{Generating Sets}
\label{generating sets}

The main goal of this section is to give specific results about splines mod $p^k$ where $p$ is prime.  We culminate by completely classifying splines over cycles mod $m$ for arbitrary integers $m$.  We begin with two simpler cases: when the base ring is $\Z/p\Z$ and when all edges in a graph are labeled  with the same ideal.

\begin{theorem}
\label{modp}
Let $p$ be a prime number.   If $G$ is an edge-labeled graph over $\Z/p\Z$ with no edges labeled zero then every vertex-labeling over $\Z/p\Z$ is a spline on $G$.
\end{theorem}

\begin{proof}

Consider an arbitrary graph $G$ in $\Z/p\Z$. Since  $\Z/p\Z$ is a field $\left<1\right>$ is the only  nonzero edge-label.    The ideal $\left<1\right>$ is the entire ring so each pair of adjacent vertices can be labeled arbitrarily in $\Z/p\Z$.  Thus any set of vertex labels will satisfy the spline conditions on $G$. 
\end{proof}

We extend the previous result to a similar result when $G$ has only one edge-label as follows.

\begin{theorem}
\label{singleedge}
 If $G$ is a connected graph such that every edge of $G$ is labeled with $\left<a\right>$  then a minimum generating set for $R_G$ is
 \begin{center}
 
 $\mathbb{B}(R_G) = \left\{\left(\begin{array}{c}1\\1\\1\\.\\.\\.\\1\\1\end{array}\right), \left(\begin{array}{c}{0}\\.\\.\\.\\0\\0\\{a}\\0\end{array}\right),\left(\begin{array}{c}{0}\\.\\.\\.\\0\\a\\{0}\\0\end{array}\right), \left(\begin{array}{c}{0}\\.\\.\\.\\a\\0\\0\\0\end{array}\right), ... , \left(\begin{array}{c}{a}\\.\\.\\.\\0\\0\\{0}\\0\end{array}\right)\right\}$
 
 \end{center}
 
 \end{theorem}

\begin{proof}
Let $G_n$ denote any connected graph on $n$ vertices for which every edge is labeled $a$.  The theorem describes a set of flow-up classes whose entries are all $a$ or $0$.  If these splines together with the identity spline generate $R_{G,\alpha}$ then they satisfy the hypotheses of Corollary \ref{corollary: flow-up with one label} and thus form a minimum generating set for $R_{G,\alpha}$.  

Thus we prove that the set in the theorem generates $R_{G_n,\alpha}$ for any $G_n$.  Our proof proceeds by induction on the number of vertices $n$.

The base case is clear (see also \cite[Theorem 4.1]{GPT}):
\begin{multicols}{2}

$R_{G_2} =$ 

\begin{center}

\begin{tikzpicture}
	\pgfmathsetmacro{\r}{2.5}
	\pgfmathsetmacro{\ro}{2.75}
	\pgfmathsetmacro{\edge}{36}

	\node (a) at ({-90 + \edge*0}:\r) {};
	\node (b) at ({-90 + \edge*1}:\r) {};
	
	\draw[edge] (a)--(b);

	\node[edgelabel] at ({-72 + \edge*0}:\ro) {$a$};

	\node[vertex] at (a) {$\color{white}..$};
	\node[vertex] at (b) {$\color{white}..$};
		\end{tikzpicture}
	
	\end{center}

\columnbreak

 $\mathbb{B}(R_{G_2}) = \left\{\left(\begin{array}{c}1\\1\end{array}\right), \left(\begin{array}{c}{a}\\{0}\end{array}\right)\right\}$

\end{multicols}

Assume $\mathbb{B}(R_{G_i})$  gives a minimum generating set for $R_{G_i}$ for any $G_i$ with $1 < i \le n-1$ and let $G_n$ be a graph with $n$ vertices.  Choose any $n-1$ of the vertices and let $G_{n-1}$ be the graph induced by those vertices.  Let $v$ denote the $n^{th}$ vertex.  Every edge in $G_n$ and hence $G_{n-1}$ is labeled $a$ so the inductive hypothesis holds for $G_{n-1}$.  Consider $\pi_{n}: R_{G_n} \rightarrow R_{G_{n-1}}$.  

The  difference between any two adjacent vertex-labels in any element of $\mathbb{B}(R_{G_{n}})$ is either $0$ or $a$.  Since every edge in $G_n$ is labeled $a$ by assumption the spline conditions are satisfied in $G_{n}$.  Thus the elements in $\mathbb{B}(R_{G_{n}})$ are all splines on $G_n$.  Notice that the elements of $\mathbb{B}(R_{G_{n-1}})$ agree with the first $n-1$ elements of $\mathbb{B}(R_{G_{n}})$ in the lowest $n-1$ entries.   Thus $\pi_{n}$ is surjective and $M_v$ is generated by $[a,0,0,...,0]^{t}$.  We conclude from Lemma \ref{lemma: adding one vertex} that $R_{G_n} = R_{G_{n-1}} \oplus M_v$ and hence $\mathbb{B}(R_{G_{n}})$ is a minimum generating set for splines on $G_n$.  
\end{proof}

The case of splines over $\Z/p^2\Z$ follows as a corollary.

\begin{corollary}
Let $G$ be a graph and $p$ a prime number.  Then splines on $G$ over $\Z/p^2\Z$ are generated by the minimum generating set $\mathbb{B}(R_G)$ in Theorem \ref{singleedge}.
\end{corollary}

\begin{proof}
The ideal $\left<p\right>$ generates all zero divisors in $\Z/{p^2}\Z$ which  means there exists only one nonzero, non-unit edge label in $\Z/p^2\Z$.  Hence graphs over $\Z/p^2\Z$ are completely characterized by this theorem.  
\end{proof}




We extend this result partially to cycles whose edges are labeled by powers of $a$.
 
\begin{theorem}
\label{multiplesofa}
Fix a zero divisor $a$ in $\Z/m\Z$.  Suppose all of the edges of $C_n$ are labeled with powers of $a$ so the set of edge labels is $\{a^{k_1}, a^{k_2}, a^{k_3}, ... , a^{k_n}\}$.  Without loss of generality assume that $a^{k_1}$ is the minimal power in the set and that $a^{k_1}$ is the label on edge $\ell_n$. Then the following set generates all splines on $C_n$.

\begin{center}

 $\mathbb{B}=  \left\{\left(\begin{array}{c}1\\1\\1\\.\\.\\.\\1\\1\end{array}\right), \left(\begin{array}{c}{\ell_1}\\\ell_1\\.\\.\\\ell_1\\ \ell_1\\{\ell_1}\\0\end{array}\right),\color{black} \left(\begin{array}{c}{\ell_2}\\\ell_2\\.\\.\\\ell_2\\\ell_2\\{0}\\0\end{array}\right) \color{black}, ... , \left(\begin{array}{c}{\ell_i}\\\ell_i\\.\\.\\\ell_i\\0\\.\\0\end{array}\right), ... , \left(\begin{array}{c}{\ell_{n-2}}\\\ell_{n-2}\\.\\.\\.\\0\\0\\0\end{array}\right),\left(\begin{array}{c}{\ell_{n-1}}\\\\.\\.\\.\\0\\0\\0\end{array}\right)\right\}$

\end{center}
\end{theorem}

\begin{proof}

We need to verify that every element in our generating set is a spline on $C_n$ and that every possible spline on $C_n$ can be written in terms of elements of our generating set.  

The trivial spline is a spline by definition.  Notice that every other element of $\mathbb{B}$ is of the form $(\ell_i, \ell_i, ... , \ell_i, 0, ..., 0)^T$.  The difference between any pair of adjacent vertices is  $0$ around every edge except around the edges $\ell_i$ and $\ell_n$.  The spline conditions are trivially satisfied for each pair of adjacent vertices that differ by $0$.  The difference over the other two edges is $\ell_i$. Notice that $\ell_i$ divides itself and recall our convention that the $n^{th}$ edge is labeled with $a^{k_1}$ which divides all other edge-labels by assumption.  Thus the spline conditions are satisfied at every edge. 

 Suppose $\spl{f}{}$ is an arbitrary spline in $R_{C_n}$. We induct on the number of leading zeros of $\spl{f}{}$ to show that it is generated by $\mathbb{B}$.  If it has no leading zeros then the linear combination $\spl{f}{} - \spl{f}{v_1} \spl{1}{}$ is a spline in $R_{C_n}$ with one leading zero.  If the spline has $i$ leading zeros then $\spl{f}{v_{i+1}}$ must be a multiple of $\ell_i$ since the $i^{th}$ edge is $v_iv_{i+1}$ and $v_{i}$ is labeled zero by hypothesis.  Let $c = \spl{f}{v_{i+1}}/\ell_i$.  It follows that $\spl{f}{}  - \spl{f}{v_{i+1}}  (\ell_i, \ell_i, ... , \ell_i, 0, ..., 0)^T$ is a spline in $R_{C_n}$ with $i+1$ leading zeros.  By induction we conclude every element $\spl{f}{} \in R_{C_n}$ can be written as a linear combination of the splines in $\mathbb{B}$.
 
 Corollary \ref{corollary: flow-up with one label} shows that this set is minimum, proving the claim.  
\end{proof}

The previous result completely characterizes the ring of splines for cycles over $\Z/p^k\Z$ as we see next.

\begin{corollary}\label{corollary: cycles and prime powers}
Let $C_n$ be the cycle on $n$ vertices, let $p$ be a prime number and let $k$ be any positive integer.  Then the splines on $C_n$ over $\Z/p^k\Z$ are generated by the minimum generating set $\mathbb{B}$ in Theorem \ref{multiplesofa}.
\end{corollary}

\begin{proof}
The only possible edge labels over $\Z/p^k \Z$ are $\{\left<p\right>, \left<p^2\right>, \left<p^3\right>, ... , \left<p^{k-1}\right>\}$.  By rotating the edge-labeled graph we can assume that the edge $\ell_n$ is labeled with the least power. This rotation induces an isomorphism on the ring of splines. Thus Theorem \ref{multiplesofa} gives a minimum generating set for $R_{{C_n},\alpha}$ over $\Z/p^k\Z$.
\end{proof}

Together these results allow us to completely describe splines for cycles over $\Z/m\Z$.

\begin{theorem} Fix $m$ with prime factorization $m = {p_1}^{k_1} {p_2}^{k_2}...{p_j}^{k^j}$. Let $C_n$ be a cycle with edge-labeling $\alpha: E(C_n) \rightarrow \textup{Ideals in }\Z/m\Z$. For each $i=1,\ldots,j$ let $R^{(i)}$ denote $\Z/p_i^{k_i}\Z$ and let $\alpha^{(i)}$ denote the edge-labeling of $C_n$ that sends each edge $e$ to $\alpha(e) \mod p_i^{k_i}$.  Then 
\[R_{{C_n},\alpha} \cong R^{(1)}_{{C_n},\alpha^{(1)}}  \oplus R^{(2)}_{{C_n},\alpha^{(2)}}  \oplus {R^{(3)}}_{{C_n},\alpha^{(3)}} \oplus ... \oplus {R^{(j)}}_{{C_n},\alpha^{(j)}}\]
where each $R^{(i)}_{{C_n},\alpha^{(i)}}$ has a minimum generating set $\mathbb{B}$ given in Theorem \ref{multiplesofa}.
\label{structuretheorem}
\end{theorem}
\begin{proof}
 Theorem \ref{structure theorem} decomposes $R_{{C_n},\alpha}$ as
\[R_{{C_n},\alpha} \cong R^{(1)}_{{C_n},\alpha^{(1)}}  \oplus R^{(2)}_{{C_n},\alpha^{(2)}}  \oplus {R^{(3)}}_{{C_n},\alpha^{(3)}} \oplus ... \oplus {R^{(j)}}_{{C_n},\alpha^{(j)}}.\] 
 Theorem \ref{multiplesofa} applies directly to each ${R^{(i)}}_{{C_n},\alpha^{(i)}}$ since each is the ring of splines of a cycle over $\Z/p_i^{k_i}\Z$.  Hence we can completely describe all splines on  $C_n$ in $\Z/m\Z$ for arbitrary $m$ and $n$.
\end{proof}

The previous strategy and Theorem \ref{structuretheorem} are theoretically powerful.  However they do not always produce the cleanest and most useful results.  The next theorem illustrates this point.  It produces a very elegant minimum generating set in a case where Theorem \ref{structuretheorem} would not directly apply (since $n_1$ and $n_2$ need not be relatively prime), and the minimum generating set it produces is not immediately the same as the one that we would get via Theorem \ref{structuretheorem}.  In fact while Theorem \ref{structuretheorem} could be used to prove the following result when $n_1,n_2$ are coprime, the proof would not be significantly shorter. So in computational problems, ad hoc strategies may still be useful.

\begin{theorem}
\label{pq1}

Let $C_n$ be labeled as in Figure \ref{ncyclelabels}. Fix $m, m_1, m_2$ such that $m_1 \not = m_2$ and $\lcm (m_1,m_2) = m$.  Assume every edge of $C_n$ is labeled with either $m_1$ or $m_2$ and that both $m_1$ and $m_2$ appear as edge labels at least once.  Then the following set $\mathbb{B}$ is a flow-up generating set for $R_{C_n}$:
   
   \begin{center}

 $\mathbb{B} =  \left\{\left(\begin{array}{c}1\\1\\1\\.\\.\\.\\1\\1\end{array}\right), \left(\begin{array}{c}{z_1}\\\ell_1\\.\\.\\\ell_1\\ \ell_1\\{\ell_1}\\0\end{array}\right),\color{black} \left(\begin{array}{c}{z_2}\\\ell_2\\.\\.\\\ell_2\\\ell_2\\{0}\\0\end{array}\right) \color{black}, ... , \left(\begin{array}{c}{z_i}\\\ell_i\\.\\.\\\ell_i\\0\\.\\0\end{array}\right), ... , \left(\begin{array}{c}{z_{n-2}}\\\ell_{n-2}\\.\\.\\.\\0\\0\\0\end{array}\right) \right\}$

   where  $ z_i = 0$  if $\ell_i = m_2  $
 and
$  z_i = \ell_i $ if $ \ell_i = m_1$ 
\end{center}

\end{theorem}

\begin{proof}

	 Without loss of generality assume that $\ell_n = m_1$ and $\ell_{n-1} = m_2$.  
We need to verify that every element in $\mathbb{B}$ is a spline on $C_n$ and that every possible spline on $C_n$ can be written in terms of elements in the set.  
 
 To demonstrate the former, note that the trivial spline is a spline on every graph.  In the non-trivial elements of $\mathbb{B}$, notice that the difference between any two adjacent vertices is $0$ over every edge  $\ell_k$ for $0<k<i$ or $i<k<n-1$.  Spline conditions are trivially satisfied in these cases.  The labels for the vertices incident to $\ell_i$ are $(0,\ell_i)$ and thus also satisfy spline conditions.  The edges $\ell_n$ and $\ell_{n-1}$ connect vertices labeled $(0, z_i)$ and $(z_i, \ell_i)$ respectively.  Note that $z_i$ is defined to be exactly that vertex label which satisfies the spline conditions for these edge labels.  Thus every element in $\mathbb{B}$ is a spline on $C_n$

Every spline on $C_n$ can be written in terms of elements in $\mathbb{B}$ by the same inductive argument that appeared in Theorem \ref{multiplesofa}.  The major difference between this case and Theorem \ref{multiplesofa} is that our generating set looses a rank: namely  if $(s, 0,..., 0,0,0)^T$ is a spline on $C_n$ with $n-1$ leading zeros then $s$ must be a multiple of $\ell_{n-1}$ and $\ell_n$.  Since $\lcm(m_1,m_2)=m$ we conclude $s =0$.  This proves the claim.
\end{proof}
  
\subsection{An algorithm to produce Minimum Generating Sets from Theorem \ref{pq1}
}
 \label{minimal generating sets}
 
We can use Theorems \ref{multiplesofa} or \ref{pq1} to find a generating set for $R_{{C_n},\alpha}$ over $\Z/m\Z$.  However these generating sets may not be minimum.  In this section we present an algorithm that turns these generating sets into a minimum generating set.  The algorithm is essentially the same as that which relates prime factorizations to elementary divisors in the structure theorem for finite abelian groups (see, e.g., \cite[Page 164]{DumFoo03}).


\begin{algorithm} Minimum Generating Set Algorithm
   \label{algorithm}

\begin{enumerate}
\item Fix $m$ and a factorization $m=m_1m_2 \cdots$.
\item Let $\mathbb{B}$ be a generating set for $R_{{C_n},\alpha}$ that consists of the identity spline together with constant flow-up splines whose order divides one of $m_1, m_2, \ldots$.
\item  Write $\spl{b'_1}{}, \spl{b'_2}{}, \ldots$ for the splines of order $m_1$, $\spl{b''_1}{}, \spl{b''_2}{}, \ldots$ for the splines of order $m_2$, and so on. 
\item For each $i$ remove as many of $\spl{b'_i}{}, \spl{b''_i}{},\ldots$ as possible from $\mathbb{B}$ and add the sum $\spl{b'_i}{}+\spl{b''_i}{}+\cdots$ to the set $\mathbb{M}$.
\item Any splines left in $\mathbb{B}$ have the same order so put them in $\mathbb{M}$ too. 
\end{enumerate}

\end{algorithm}

\begin{example} 
\label{making generating sets minimal}
 Here is an example of how to use Algorithm \ref{algorithm}.  
 
 \begin{enumerate}
 
\item We use $\Z/21\Z$ and $C_6$  as in Figure 2.

\begin{figure}[h]
\begin{tikzpicture}
	\pgfmathsetmacro{\r}{2.5}
	\pgfmathsetmacro{\ro}{2.75}
	\pgfmathsetmacro{\edge}{36}

	\node (a) at ({-90 + \edge*0}:\r) {};
	\node (b) at ({-90 + \edge*1}:\r) {};
	\node (c) at ({-90 + \edge*2}:\r) {};
	\node (d) at ({-90 + \edge*3}:\r) {};
	\node (e) at ({-90 + \edge*4}:\r) {};
	\node (f) at ({-90 + \edge*5}:\r) {};

	\draw[edge] (a)--(b);
	\draw[edge] (b)--(c);
	\draw[edge] (c)--(d);
	\draw[edge] (d)--(e);
	\draw[edge] (e)--(f);
	\draw[edge] (f)--(a);

	\node[edgelabel] at ({-72 + \edge*0}:\ro) {$3$};
	\node[edgelabel,right] at ({-72 + \edge*1}:\r) {$3$};
	\node[edgelabel] at ({-72 + \edge*2}:\ro) {$7$};
	\node[edgelabel,right] at ({-72 + \edge*3}:\r){$7$};
	\node[edgelabel] at ({-72 + \edge*4}:\ro) {$3$};
	\node[edgelabel] at (180:\ro /5) {$7$};

	\node[vertex] at (a) {$\color{white}..$};
	\node[vertex] at (b) {$\color{white}..$};
	\node[vertex] at (c) {$\color{white}..$};
	\node[vertex] at (d) {$\color{white}..$};
	\node[vertex] at (e) {$\color{white}..$};
	\node[vertex] at (f) {$\color{white}..$};
	
	\end{tikzpicture}
\caption{Every edge is labeled $3$ or $7$ and $3 \cdot 7 \equiv 0 \mod 21$}

\end{figure}
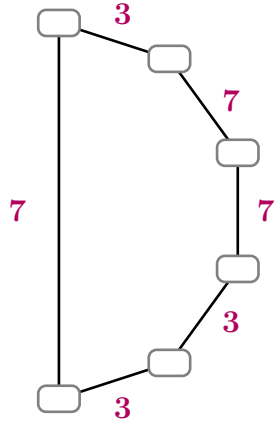

\item  We obtain the nonminimum generating set $\mathbb{B}$.  
\begin{center}
 $\mathbb{B} =  \left\{\left(\begin{array}{c}1\\1\\1\\1\\1\\1\end{array}\right), \left(\begin{array}{c}{0}\\3\\3\\3\\3\\ 0\end{array}\right),\color{black} \left(\begin{array}{c}{0}\\3\\3\\3\\{0}\\0\end{array}\right), \left(\begin{array}{c}{7}\\7\\7\\0\\0\\0\end{array}\right), \left(\begin{array}{c}{7}\\7\\0\\0\\0\\0\end{array}\right) \right\}$

\end{center}
Every element in $\mathbb{B}$ generates an additive subgroup isomorphic to a subgroup of $\Z/21\Z$.  The trivial spline $\spl{1}{}$ is the only element of $\mathbb{B}$ that generates the whole subgroup $\Z/21\Z$.
\\

\item Make as large a matching as possible of splines of additive order $3$ together with splines of additive order $7$:
\\

\begin{center}
  $\left\{\left(\begin{array}{c}{0}\\3\\3\\3\\3\\ 0\end{array}\right),  \left(\begin{array}{c}{7}\\7\\7\\0\\0\\0\end{array}\right)\right\}$,   $ \left\{\left(\begin{array}{c}{0}\\3\\3\\3\\{0}\\0\end{array}\right), \left(\begin{array}{c}{7}\\7\\0\\0\\0\\0\end{array}\right)\right\}$
\end{center}

\item
Sum the pairs from the previous step to form new flow-up splines:

\begin{center}
$\left(\begin{array}{c}{7}\\10\\10\\3\\3\\ 0\end{array}\right)$ , $\left(\begin{array}{c}{7}\\10\\3\\3\\0\\{0}\end{array}\right)$
\end{center}

\item
The minimum generating set $\mathbb{M}$ consists of the two new splines together with any unpaired splines from the original generating set. 

\begin{center}
$\mathbb{M} = \left\{\left(\begin{array}{c}1\\1\\1\\1\\1\\1\end{array}\right), \left(\begin{array}{c}{7}\\10\\10\\3\\3\\ 0\end{array}\right), \left(\begin{array}{c}{7}\\10\\3\\3\\0\\{0}\end{array}\right)\right\}  $ 
\end{center}

\end{enumerate}

\end{example}

\begin{theorem}\label{theorem: algorithm works}
Suppose that the generating set $\mathbb{B}$ for $R_{{C_n},\alpha}$ is obtained either from the pullbacks of the minimum generating sets in Theorem \ref{structuretheorem} or from Theorem \ref{pq1}.  Then the set $\mathbb{M}$ is a minimum generating set for $R_{{C_n},\alpha}$.
\end{theorem}

\begin{proof}
In the case of Theorem \ref{structuretheorem} we have a factorization $m=\prod p_i^{k_i}$ into distinct prime powers.  For Theorem \ref{pq1} we have $m_1$ and $m_2$ with $\lcm(m_1,m_2)=m$.  We use the relatively prime pair $m_1$ and $\frac{m_2}{\gcd(m_1,m_2)}$ for which also 
\[m_1 \cdot \frac{m_2}{\gcd(m_1,m_2)} = \lcm(m_1,m_2)=m.\]
In both cases consider the quotient maps $\Z/m\Z \rightarrow \Z/d_i\Z$ for each divisor $d_i|m$ as well as the induced maps on splines $R_{C_n,\alpha} \rightarrow R^{(i)}_{C_n,\alpha^{(i)}}$.  By Theorem \ref{structure theorem} we know that 
\[R_{C_n,\alpha} \cong R^{(1)}_{C_n,\alpha^{(1)}} \oplus R^{(2)}_{C_n,\alpha^{(2)}} \oplus \cdots \]
In the case of Theorem \ref{pq1} the function $\alpha^{(i)}$ labels $j$ edges with $0$ and the other $n-j$ edges with $m_i$.  By Remark \ref{unitsandzeros} and Theorem \ref{singleedge} we conclude that $R^{(i)}_{C_n,\alpha^{(i)}}$ has rank $n-j$.  This is also the size of the image of $\mathbb{B}$ in $R^{(i)}_{C_n,\alpha^{(i)}}$ so $\mathbb{B}$ consists of the preimage of a minimum generating set for $R^{(1)}_{C_n,\alpha^{(1)}}$ and $R^{(2)}_{C_n,\alpha^{(2)}}$.  (This is true by hypothesis for Theorem \ref{structure theorem}.)

 If the order of an element $\spl{b}{} \in R_{C_n,\alpha}$ is $d$ then the order of its image in $R^{(i)}_{C_n,\alpha^{(i)}}$ divides $d$ for each $i$. In the case of Theorem \ref{structuretheorem}  the order of each generator divides $p_i^{k_i}$ for some $i$ so the image of $\spl{b}{} \in \mathbb{B}$ is zero in all direct summands except $R^{(i)}_{C_n,\alpha^{(i)}}$.  In the case of Theorem \ref{pq1} we constructed the quotient maps so that if $\spl{b}{}  \in \mathbb{B}$ has order $m_i$ then its image in $R^{(i)}_{C_n,\alpha^{(i)}}$ is zero for each $i=1,2$.  As  proven in the previous paragraph, when the image of $\spl{b}{} \in \mathbb{B}$ is nonzero then it is part of a minimum generating set for a unique $R^{(i)}_{C_n,\alpha^{(i)}}$.  Thus the sums in the set $\mathbb{M}$ are precisely the elements identified in the invariant factor decomposition (see Corollary \ref{corollary: invariant factor decomposition}), so they form a minimum generating set for $R_{C_n,\alpha}$ as desired. 
\end{proof}

\section{Questions}
\label{questions section}

Our first collection of questions extends the work in this paper directly.
 
 Answering this next question would allow us to classify splines for larger families of graphs over $\Z/m\Z$.
 
\begin{question}
Can we find a minimum generating set of constant flow-up classes over $\Z/p^k\Z$?  Can we classify splines for other families of graphs over $\Z/p^k\Z$?
\end{question}

Theorem \ref{structure theorem} might be exploited to produce more explicit descriptions of splines over $\Z/m\Z$ if $m$ has few prime factors.

\begin{question}
 Can we classify splines over $\Z/p_1p_2\Z$ for $p_1,p_2$ prime?  Or over $\Z/p_1p_2p_3 \Z$ for $p_1,p_2,p_3$ prime?
\end{question}

The classifications that we have completed, together with the questions above, lead us to wonder about combinatorial statistics of the module of splines.  For instance we have the following.

\begin{question}
Given $G$ and $m$ what is the smallest possible rank of $R_{G,\alpha}$?  What is the distribution of $\rk R_{G,\alpha}$ over all edge-labeling functions $\alpha$?
\end{question}

The structure theorems that we have developed lead to the following kind of classification, as well, in which we classify spline modules themselves rather than identify the possible splines for a given graph or ring.

\begin{question}
Can we classify isomorphism classes of splines given $m, n$?  For instance if $n,m$ are fixed how many distinct rings $R_{G,\alpha}$ arise as spline modules for a graph with $n$ vertices over $\Z/m\Z$?
\end{question}

We generally restrict our attention in this paper to rings of the form $\Z/m\Z$ though the arguments in Theorem \ref{structure theorem} and Corollary \ref{corollary: invariant factor decomposition} could be extended to other rings.   

\begin{question} 
How broadly can Theorem \ref{structure theorem} and Corollary \ref{corollary: invariant factor decomposition} be extended?  Can they be used to develop effective computational tools like Algorithm \ref{algorithm}?
\end{question}

Finally objects similar to splines over quotient rings arise naturally in Braden-MacPherson's construction of intersection homology \cite{BraMac01}.  In that case we have quotients of polynomial rings (and sums of quotients of polynomial rings) rather than integers.  

\begin{question}
 Which of the results in this paper extend to the Braden-MacPherson setting?
\end{question}

\section{Acknowledgements}
\label{sec:ack}
The authors gratefully acknowledge useful conversations with Michael DiPasquale, Elizabeth Drellich, Sarah Hagen, Melanie King, Stephanie Reinders, and Lauren Rose as well as the support of the Smith College Center for Women in Mathematics.


\begin{thebibliography}{10}

\bibitem{Alf86}
P.~Alfeld.
\newblock On the dimension of multivariate piecewise polynomials.
\newblock In {\em Numerical analysis ({D}undee, 1985)}, volume 140 of {\em
  Pitman Res. Notes Math. Ser.}, pages 1--23. Longman Sci. Tech., Harlow, 1986.

\bibitem{AlfSch87}
Peter Alfeld and L.~L. Schumaker.
\newblock The dimension of bivariate spline spaces of smoothness {$r$} for
  degree {$d\geq 4r+1$}.
\newblock {\em Constr. Approx.}, 3(2):189--197, 1987.

\bibitem{AlfSch90}
Peter Alfeld and L.~L. Schumaker.
\newblock On the dimension of bivariate spline spaces of smoothness {$r$} and
  degree {$d=3r+1$}.
\newblock {\em Numer. Math.}, 57(6-7):651--661, 1990.

\bibitem{BFR09}
Anthony Bahri, Matthias Franz, and Nigel Ray.
\newblock The equivariant cohomology ring of weighted projective space.
\newblock {\em Math. Proc. Cambridge Philos. Soc.}, 146(2):395--405, 2009.

\bibitem{Bil88}
Louis~J. Billera.
\newblock Homology of smooth splines: generic triangulations and a conjecture
  of {S}trang.
\newblock {\em Trans. Amer. Math. Soc.}, 310(1):325--340, 1988.

\bibitem{BilRos91}
Louis~J. Billera and Lauren~L. Rose.
\newblock A dimension series for multivariate splines.
\newblock {\em Discrete Comput. Geom.}, 6(2):107--128, 1991.

\bibitem{BilRos92}
Louis~J. Billera and Lauren~L. Rose.
\newblock Modules of piecewise polynomials and their freeness.
\newblock {\em Math. Z.}, 209(4):485--497, 1992.

\bibitem{BHKR}
Nealy Bowden, Sarah Hagen, Melanie King, and Stephanie Reinders.
\newblock The ring of integer splines on cycles.
\newblock {\em insert arxiv page}.

\bibitem{BraMac01}
Tom Braden and Robert MacPherson.
\newblock From moment graphs to intersection cohomology.
\newblock {\em Math. Ann.}, 321(3):533--551, 2001.

\bibitem{DiP12}
Michael~R. DiPasquale.
\newblock Shellability and freeness of continuous splines.
\newblock {\em J. Pure Appl. Algebra}, 216(11):2519--2523, 2012.

\bibitem{DiP14}
Michael~R. DiPasquale.
\newblock Lattice-supported splines on polytopal complexes.
\newblock {\em Adv. in Appl. Math.}, 55:1--21, 2014.

\bibitem{Dou51b}
Jesse Douglas.
\newblock On the {B}asis {T}heorem for finite abelian groups (second note).
\newblock {\em Proc. Natl. Acad. Sci. USA}, 37(8):525--528, 1951.

\bibitem{Dou51c}
Jesse Douglas.
\newblock On the {B}asis {T}heorem for finite abelian groups (third note).
\newblock {\em Proc. Natl. Acad. Sci. USA}, 37(9):611--614, 1951.

\bibitem{Dou51a}
Jesse Douglas.
\newblock On the existence of a basis for every finite abelian group.
\newblock {\em Proc. Natl. Acad. Sci. USA}, 37(6):359--362, 1951.

\bibitem{Dou51d}
Jesse Douglas.
\newblock On the invariants of finite abelian groups.
\newblock {\em Proc. Natl. Acad. Sci. USA}, 37(10):672--677, 1951.

\bibitem{DumFoo03}
David Dummit and Richard Foote.
\newblock {\em Abstract Algebra}.
\newblock Wiley, Hoboken, NJ, 3 edition, 2003.

\bibitem{GPT}
Simcha Gilbert, Shira Polster, and Julianna Tymoczko.
\newblock Rings of generalized splines.
\newblock {\em arXiv:1306.0801}.

\bibitem{GolTol09}
Rebecca~F. Goldin and Susan Tolman.
\newblock Towards generalizing {S}chubert calculus in the symplectic category.
\newblock {\em J. Symplectic Geom.}, 7(4):449--473, 2009.

\bibitem{GKM98}
Mark Goresky, Robert Kottwitz, and Robert MacPherson.
\newblock Equivariant cohomology, {K}oszul duality, and the localization
  theorem.
\newblock {\em Invent. Math.}, 131(1):25--83, 1998.

\bibitem{Haa91}
Ruth Haas.
\newblock Module and vector space bases for spline spaces.
\newblock {\em J. Approx. Theory}, 65(1):73--89, 1991.

\bibitem{HagTym}
Sarah Hagen and Julianna Tymoczko.
\newblock A constructive algorithm to find a basis for splines over principal
  rings and {P}r\"{u}fer domains.

\bibitem{HMR}
Madeleine Handschy, Julie Melnick, and Stephanie Reinders.
\newblock Integer generalized splines on cycles.
\newblock {\em arXiv:1409.1481}.

\bibitem{KnuTao03}
Allen Knutson and Terence Tao.
\newblock Puzzles and (equivariant) cohomology of {G}rassmannians.
\newblock {\em Duke Math. J.}, 119(2):221--260, 2003.

\bibitem{McN}
George McNulty.
\newblock Algebra lecture notes, {L}ecture 5: Submodules of free modules over a
  {PID}.
\newblock {\em Available at
  http://people.math.sc.edu/mcnulty/algebra/grad/pidfree.pdf}.

\bibitem{Pay06}
Sam Payne.
\newblock Equivariant {C}how cohomology of toric varieties.
\newblock {\em Math. Res. Lett.}, 13(1):29--41, 2006.

\bibitem{Ros}
Lauren Rose.
\newblock Personal communication with the second author.
\newblock 2013.

\bibitem{Ros95}
Lauren~L. Rose.
\newblock Combinatorial and topological invariants of modules of piecewise
  polynomials.
\newblock {\em Adv. Math.}, 116(1):34--45, 1995.

\bibitem{Ros04}
Lauren~L. Rose.
\newblock Graphs, syzygies, and multivariate splines.
\newblock {\em Discrete Comput. Geom.}, 32(4):623--637, 2004.

\bibitem{Sch12}
Hal Schenck.
\newblock Equivariant {C}how cohomology of nonsimplicial toric varieties.
\newblock {\em Trans. Amer. Math. Soc.}, 364(8):4041--4051, 2012.

\bibitem{Sch84b}
Larry~L. Schumaker.
\newblock Bounds on the dimension of spaces of multivariate piecewise
  polynomials.
\newblock {\em Rocky Mountain J. Math.}, 14(1):251--264, 1984.
\newblock Surfaces (Stanford, Calif., 1982).

\bibitem{Sch84a}
Larry~L. Schumaker.
\newblock On spaces of piecewise polynomials in two variables.
\newblock In {\em Approximation theory and spline functions ({S}t.\ {J}ohn's,
  {N}fld., 1983)}, volume 136 of {\em NATO Adv. Sci. Inst. Ser. C Math. Phys.
  Sci.}, pages 151--197. Reidel, Dordrecht, 1984.

\bibitem{Tym05}
Julianna~S. Tymoczko.
\newblock An introduction to equivariant cohomology and homology, following
  {G}oresky, {K}ottwitz, and {M}ac{P}herson.
\newblock In {\em Snowbird lectures in algebraic geometry}, volume 388 of {\em
  Contemp. Math.}, pages 169--188. Amer. Math. Soc., Providence, RI, 2005.

\bibitem{Yuz92}
Sergey Yuzvinsky.
\newblock Modules of splines on polyhedral complexes.
\newblock {\em Math. Z.}, 210(2):245--254, 1992.

\end{thebibliography}
\end{document}